\documentclass[8pt,a4paper]{article}
\usepackage[utf8]{inputenc}
\usepackage{amsfonts}
\usepackage{amsthm}
\usepackage{amsmath}
\usepackage{amssymb}
\usepackage{enumitem,kantlipsum}

\newtheorem{theorem}{Theorem}[subsection]
\theoremstyle{definition}
\newtheorem{definition}[theorem]{Definition}
\newtheorem{corollary}[theorem]{Corollary}
\newtheorem{lemma}[theorem]{Lemma}
\newtheorem{proposition}[theorem]{Proposition}

\newtheorem{example}[theorem]{Example}
\newtheorem{remark}[theorem]{Remark}

\def\b{\bold{B}}

\def\p2{$PSL_2(q)$}

\def\b2{B_{n-2}}

\begin{document}

\author{Sawsan Khaskeia\\Department of Mathematics\\Ariel University, Israel\\sawsan@ariel.ac.il\and Robert Shwartz\\Department of Mathematics\\Ariel University, Israel\\robertsh@ariel.ac.il}
\date{}
\title{Generalization of the basis theorem for the $D$-type Coxeter groups}
\maketitle
    
  \begin{abstract}
The $OGS$ for  non-abelian groups is an interesting
generalization of the basis of finite abelian groups. The definition of $OGS$ states that every element of a group has a unique presentation as a product of some powers of specific generators of the group, in a specific given order. In case of the symmetric groups $S_{n}$ there is a paper of R. Shwartz, which demonstrates a strong connection between the $OGS$ and the standard Coxeter presentation of the symmetric group, which is called the standard $OGS$ of $S_n$. In this paper we generalize the standard $OGS$ of $S_n$  to the finite classical Coxeter group $D_n$. We describe the exchange laws for the generalized standard $OGS$ of $D_n$, and we connect it to the Coxeter length of elements of $D_n$. 
\end{abstract}

\section{Introduction}
The fundamental theorem of finitely generated abelian groups states the following:
Let $A$ be a finitely generated abelian group, then there exists generators $a_{1}, a_{2}, \ldots a_{n}$, such that every element $a$ in $A$ has a unique presentation of a form:
$$g=a_{1}^{i_{1}}\cdot a_{2}^{i_{2}}\cdots a_{n}^{i_{n}},$$
where, $i_{1}, i_{2}, \ldots, i_{n}$ are $n$ integers such that for  $1\leq k\leq n$, $0\leq i_{k}<|g_{k}|$, where $a_{k}$ has a finite order of $|a_{k}|$ in $A$, and $i_{k}\in \mathbb{Z}$, where $a_{k}$ has infinite order in $A$.
Where, the meaning of the theorem is that every abelian group $A$ is direct sum of finitely many cyclic subgroup $A_{i}$ (where $1\leq i\leq k$), for some $k\in \mathbb{N}$.

\begin{definition}\label{ogs}
Let $G$ be a non-abelian group. The ordered sequence of $n$ elements $\langle g_{1}, g_{2}, \ldots, g_{n}\rangle$ is called an $Ordered ~~ Generating ~~ System$ of the group $G$ or by shortened notation, $OGS(G)$, if every element $g\in G$ has a unique presentation in the a form
$$g=g_{1}^{i_{1}}\cdot g_{2}^{i_{2}}\cdots g_{n}^{i_{n}},$$
where, $i_{1}, i_{2}, \ldots, i_{n}$ are $n$ integers such that for  $1\leq k\leq n$, $0\leq i_{k}<r_{k}$, where  $r_{k} | |g_{k}|$  in case the order of $g_{k}$ is finite in $G$, or   $i_{k}\in \mathbb{Z}$, in case  $g_{k}$ has infinite order in $G$.
The mentioned canonical form is called $OGS$ canonical form.  For every $q>p$, $1\leq x_{q}<r_{q}$, and $1\leq x_{p}<r_{p}$ the relation
$$g_{q}^{x_{q}}\cdot g_{p}^{x_{p}} = g_{1}^{i_{1}}\cdot g_{2}^{i_{2}}\cdots g_{n}^{i_{n}},$$
is called exchange law.
\end{definition}
In contrast to finitely generated abelian groups, the existence of an $OGS$ is generally not true for every finitely generated non-abelian group. Even in case of two-generated infinite non-abelian groups it is not difficult to find counter examples. For example, the Baumslag-Solitar groups $BS(m,n)$ \cite{BS}, where $m\neq \pm1$ or $n\neq \pm1$, or most of the cases of the one-relator free product of a finite cyclic group generated by $a$, with a finite two-generated group generated by $b, c$ with the relation $a^{2}\cdot b\cdot a\cdot c=1$ \cite{S}, do not have an $OGS$. Even the question of the existence of an $OGS$ for a general finite non-abelian group is still open. Moreover, contrary to the abelian case where the exchange law is just $g_{q}\cdot g_{p}=g_{p}\cdot g_{q}$, in most of the cases of non-abelian groups with the existence of an $OGS$, the exchange laws are very complicated.  Although there are some specific non-abelian groups where the exchange laws are very convenient and have very interesting properties. A very good example of it is the symmetric group $S_{n}$. In 2001, Adin and Roichman \cite{AR} introduced a presentation of an $OGS$ canonical form  for the symmetric group $S_n$, for the hyperoctahedral group $B_n$, and for the wreath product $\mathbb{Z}_{m}\wr S_{n}$. Adin and Roichman proved that for every element of $S_n$ presented in the standard $OGS$ canonical form, the sum of the exponents of the $OGS$ equals the major-index of the permutation. Moreover, by using an $OGS$ canonical form, Adin and Roichman generalized the theorem of MacMahon \cite{Mac} to the $B$-type Coxeter group, and to the wreath product $\mathbb{Z}_{m}\wr S_{n}$. A few years later, that $OGS$ canonical form was generalized for complex reflection groups by Shwartz, Adin and Roichman \cite{SAR}. Recently, Shwartz \cite{S1} significantly extended the results of \cite{AR}, \cite{SAR}, where the $OGS$ of $S_{n}$ is strongly connected to the Coxeter length and to the descent set of the elements. Moreover, in \cite{S1}, there are described the exchange laws for the $OGS$ canonical forms of the symmetric group $S_n$, which have very interesting and surprising properties.
In the paper we try to generalize the results of \cite{S1} to the finite classical Coxeter group $D_{n}$. Similarly to the symmetric group $S_n$, the group $D_n$ can be considered as permutation group as well. 
Therefore, we recall the notations of permutations,  the $OGS$ of $S_{n}$ and the corresponding exchange laws, from \cite{S1}.

\begin{definition}\label{sn}
Let $S_n$ be the symmetric group on $n$ elements, then :
\begin{itemize}
\item The symmetric group $S_n$ is an $n-1$ generated simply-laced finite Coxeter group  of order $n!$,  which has the presentation of: $$\langle s_1, s_2, \ldots, s_{n-1} | s_i^{2}=1, ~~ (s_i\cdot s_{i+1})^{3}=1, ~~(s_i\cdot s_j)^2=1 ~~for ~~|i-j|\geq 2\rangle;$$
\item The group $S_n$ can be considered as the permutation group on $n$ elements. A permutation $\pi\in S_n$ is denoted by $$
\pi=[\pi(1); ~\pi(2); \ldots; ~\pi(n)]
$$
(i.e., $\pi=[2; ~4; ~1; ~3]$ is a permutation in $S_{4}$ which satisfies $\pi(1)=2$, $\pi(2)=4$, $\pi(3)=1$, and $\pi(4)=3$);
\item Every permutation $\pi\in S_n$ can be presented in a cyclic notation, as a product of disjoint cycles of the form $(i_1, ~i_2, ~\ldots, ~i_m)$, which means $\pi(i_{k})=i_{k+1}$, for $1\leq k\leq m-1$, and $\pi(i_{m})=i_{1}$
    (i.e., The cyclic notation of $\pi=
[3; ~4; ~1; ~5; ~2]$ in $S_5$, is $(1, ~3)(2, ~4, ~5)$);
\item The Coxeter generator $s_i$ can be considered the permutation which exchanges the element $i$ with the element $i+1$, i.e., the transposition $(i, i+1)$;
\item We consider multiplication of permutations in left to right order; i.e., for every $\pi_1, \pi_2\in S_n$, $\pi_1\cdot \pi_2 (i)=\pi_2(j)$, where, $\pi_1(i)=j$ (in contrary to the notation in \cite{AR} where Adin, Roichman, and other people have considered right to left multiplication of permutations);
\item For every permutation $\pi\in S_n$, the Coxeter length $\ell(\pi)$ is the number of inversions in $\pi$, i.e., the number of different pairs $i, j$, s. t. $i<j$ and $\pi(i)>\pi(j)$;
\item For every permutation $\pi\in S_n$, the set of the locations of the descents is defined to be $$des\left(\pi\right)=\{1\leq i\leq n-1 | \pi(i)>\pi(i+1)\},$$ and $$i\in des\left(\pi\right) ~~if ~and ~only ~if ~~\ell(s_i\cdot \pi)<\ell(\pi)$$ (i.e., $i$ is a descent of $\pi$ if and only if multiplying $\pi$ by $s_i$ in the left side shortens the Coxeter length of the element.);
\item For every permutation $\pi\in S_n$, the major-index is defined to be  $$maj\left(\pi\right)=\sum_{\pi(i)>\pi(i+1)}i$$ (i.e., major-index is the sum of the locations of the descents of $\pi$.).
\item By \cite{BB} Chapter 3.4, every element $\pi$ of $S_n$ can be presented uniquely in the following normal reduced form, which we denote by $norm(\pi)$:
$$norm(\pi)=\prod_{u=1}^{n-1}\prod_{r=0}^{y_{u}-1}s_{u-r}.$$
such that $y_u$ is a non-negative integer where, $0\leq y_u\leq u$ for every $1\leq u\leq n-1$.
By our notation of $norm(\pi)$ the Coxeter length of an element $\pi$ as follow: $$\ell(\pi)=\sum_{u=1}^{n-1}y_{u}.$$
For example: Let  $m=8$, $y_2=2$, $y_4=3$, $y_5=1$, $y_8=4$, and $y_1=y_3=y_6=y_7=0$, then
$$norm(\pi)=(s_2\cdot s_1)\cdot (s_4\cdot s_3\cdot s_2)\cdot s_5\cdot (s_8\cdot s_7\cdot s_6\cdot s_5).$$
$$\ell(\pi)=2+3+1+4=10.$$
\end{itemize}
\end{definition}

\begin{theorem}\label{canonical-sn}
Let $S_n$ be the symmetric group on $n$ elements. For every $2\leq m\leq n$, define $t_{m}$ to be the product $\prod_{j=1}^{m-1}s_{j}$. The element $t_{m}$ is the permutation $$t_m=
[m; ~1; ~2; \ldots; ~m-1; ~m+1;\ldots; ~n],$$
 which is the $m$-cycle $(m, ~m-1, ~\ldots, ~1)$ in the cyclic notation of the permutation. Then, the elements $t_{n}, t_{n-1}, \ldots, t_{2}$ generates $S_n$, and every element of $S_n$ has a unique presentation in the following $OGS$ canonical form:

$$t_{2}^{i_{2}}\cdot t_{3}^{i_{3}}\cdots t_{n}^{i_{n}},~~~ where ~~~0\leq i_{k}<k ~~~for ~~~2\leq k\leq n$$
\end{theorem}

\begin{proposition}\label{exchange}
The following holds:
\\

In order to transform  the element $t_{q}^{i_{q}}\cdot t_{p}^{i_{p}}$  ($p<q$) onto the $OGS$ canonical form\\ $t_{2}^{i_{2}}\cdot t_{3}^{i_{3}}\cdots t_{n}^{i_{n}}$, i.e., according to the  standard $OGS$, one needs to use the following exchange laws:

 \[ t_{q}^{i_{q}}\cdot t_{p}^{i_{p}}=\begin{cases}
t_{i_{q}+i_{p}}^{i_q}\cdot t_{p+i_{q}}^{i_{p}}\cdot t_{q}^{i_{q}}  & q-i_{q}\geq p \\
\\
t_{i_{q}}^{p+i_{q}-q}\cdot t_{i_{q}+i_{p}}^{q-p}\cdot t_{q}^{i_{q}+i_{p}} & i_{p}\leq q-i_{q}\leq p \\
\\
t_{p+i_{q}-q}^{i_{q}+i_{p}-q}\cdot t_{i_{q}}^{p-i_{p}}\cdot t_{q}^{i_{q}+i_{p}-p}  & q-i_{q}\leq i_{p}.
\end{cases}
\]

\end{proposition}

\begin{remark}\label{exchange-2}
The standard $OGS$ canonical form of $t_{q}^{i_{q}}\cdot t_{p}^{i_{p}}$ is a product of non-zero powers of two different canonical generators if and only if $q-i_{q}=p$ or $q-i_{q}=i_{p}$, as follow:
\begin{itemize}
    \item If $q-i_q=p$ then by considering  $q-i_q\geq p$: $$t_{i_{q}+i_{p}}^{i_q}\cdot t_{p+i_{q}}^{i_{p}}\cdot t_{q}^{i_{q}}=t_{i_{q}+i_{p}}^{i_q}\cdot t_{q}^{i_p}\cdot t_{q}^{i_q}$$ and by considering  $q-i_q\leq p$:  $$t_{i_{q}}^{p+i_{q}-q}\cdot t_{i_{q}+i_{p}}^{q-p}\cdot t_{q}^{i_{q}+i_{p}}=t_{i_q}^{0}\cdot t_{i_q+i_p}^{i_q}\cdot t_{q}^{i_{q}+i_{p}};$$
    \item If $q-i_q=i_p$ then by considering  $q-i_q\geq i_p$: $$t_{i_q}^{p+i_q-q}\cdot t_{i_{q}+i_{p}}^{q-p}\cdot t_{q}^{i_{q}+i_{p}}=t_{i_q}^{p-i_p}\cdot t_q^{q-p}\cdot t_q^q=t_{i_q}^{p-i_p}\cdot t_q^{q-p}$$ and by considering  $q-i_q\leq i_p$: $$t_{p+i_{q}-q}^{i_{q}+i_{p}-q}\cdot t_{i_{q}}^{p-i_{p}}\cdot t_{q}^{i_{q}+i_{p}-p}=t_{p+i_q-q}^0\cdot t_{i_{q}}^{p-i_{p}}\cdot t_{q}^{q-p}.$$
\end{itemize}
Hence we have
 \[ t_{q}^{i_{q}}\cdot t_{p}^{i_{p}}=\begin{cases}
  t_{i_{q}+i_{p}}^{i_q}\cdot t_{q}^{i_{q}+i_{p}} & q-i_{q} = p \\
\\
  t_{i_{q}}^{p-i_{p}}\cdot t_{q}^{q-p} & q-i_{q} = i_{p}.
\end{cases}
\]
\end{remark}

Moreover, the most significant achievement of the paper \cite{S1} is the definition of the standard $OGS$ elementary factorization. By using the standard $OGS$ elementary factorization, it is possible to give a very interesting formula for the Coxeter length and a complete classification of the descent set of any arbitrary element of $S_n$. In the paper we try to generalize the standard $OGS$ elementary factorization to  the $B$-type Coxeter groups, in order to find similar properties (Coxeter length and the descent set) for the elements of the group.
Hence, we recall the definition of the standard $OGS$ elementary element and factorization for the symmetric group $S_n$ as it is defined in \cite{S1}, and theorems concerning the Coxeter length and the descent set of elements of $S_n$ as it is mentioned and proved in \cite{S1}  .

\begin{definition}\label{elementary}
Let $\pi\in S_n$, where $\pi=\prod_{j=1}^{m}t_{k_{j}}^{i_{k_{j}}}$ is presented in the standard $OGS$ canonical form, with $i_{k_{j}}>0$ for every $1\leq j\leq m$. Then, $\pi$ is called standard $OGS$ elementary element of $S_n$, if
$$\sum_{j=1}^{m}i_{k_{j}}\leq k_{1}.$$
\end{definition}

\begin{theorem}\label{theorem-elementary}\cite{S1}
Let $\pi=\prod_{j=1}^{m}t_{k_{j}}^{i_{k_{j}}}$ be a standard $OGS$ elementary element of $S_n$, presented in the standard $OGS$ canonical form, with $i_{k_{j}}>0$ for every $1\leq j\leq m$. Then, the following are satisfied:
 \begin{itemize}
\item $$\ell(\pi)=\sum_{j=1}^{m}k_{j}\cdot i_{k_{j}}-(i_{k_{1}}+i_{k_{2}}+\cdots +i_{k_{m}})^{2}=\sum_{j=1}^{m}k_{j}\cdot i_{k_{j}}-\left(maj\left(\pi\right)\right)^{2};$$
\item Every subword of $\pi$ is a standard $OGS$ elementary element too. In particular, for every two subwords $\pi_{1}$ and $\pi_{2}$ of $\pi$, such that $\pi=\pi_{1}\cdot \pi_{2}$, it is satisfied: $$\ell(\pi)=\ell(\pi_{1}\cdot \pi_{2})<\ell(\pi_{1})+\ell(\pi_{2});$$
\item $$\ell(s_r\cdot \pi)=\begin{cases} \ell(\pi)-1 & r=\sum_{j=1}^{m}i_{k_{j}} \\
\ell(\pi)+1 & r\neq \sum_{j=1}^{m}i_{k_{j}} \end{cases}.$$
i.e., $des\left(\pi\right)$ contains just one element, which means $des\left(\pi\right)=\{maj\left(\pi\right)\}$.
\end{itemize}
\end{theorem}

\begin{definition}\label{canonical-factorization-def}
Let $\pi\in S_n$. Let $z(\pi)$ be the minimal number, such that $\pi$ can be presented as a product of standard $OGS$ elementary elements, with the following conditions:
\begin{itemize}
\item $$\pi=\prod_{v=1}^{z(\pi)}\pi^{(v)}, ~~~~ where ~~~~\pi^{(v)}=\prod_{j=1}^{m^{(v)}}t_{h^{(v)}_{j}}^{\imath_{j}^{(v)}},$$
 by the presentation in the standard $OGS$ canonical form  for every $1\leq v\leq z(\pi)$ and  $1\leq j\leq m^{(v)}$ such that:
 \begin{itemize}
\item $\imath_{j}^{(v)}>0;$ \\
\item $\sum_{j=1}^{m^{(1)}}\imath_{j}^{(1)}\leq h^{(1)}_{1}$ i.e., $maj\left(\pi^{(1)}\right)\leq h^{(1)}_{1}$; \\
\item $h^{(v-1)}_{m^{(v-1)}}\leq\sum_{j=1}^{m^{(v)}}\imath_{j}^{(v)}\leq h^{(v)}_{1}$ for $2\leq v\leq z$ \\ \\
i.e., $h^{(v-1)}_{m^{(v-1)}}\leq maj\left(\pi^{(v)}\right)\leq h^{(v)}_{1} ~~ for ~~ 2\leq v\leq z$.
\end{itemize}
\end{itemize}
Then, the mentioned presentation is called \textbf{Standard $OGS$ elementary factorization} of $\pi$. Since
the factors $\pi^{(v)}$ are standard $OGS$ elementary elements, they are called standard $OGS$ elementary factors of $\pi$.
\end{definition}

\begin{theorem}\label{theorem-factorization}\cite{S1}
Let $\pi=\prod_{j=1}^{m}t_{k_{j}}^{i_{k_{j}}}$ be an element of $S_n$ presented in the standard $OGS$ canonical form, with $i_{k_{j}}>0$ for every $1\leq j\leq m$. Consider the standard $OGS$ elementary factorization of $\pi$ with all the notations used in Definition \ref{canonical-factorization-def}.
Then, the following properties hold:
\begin{itemize}
\item The standard $OGS$ elementary factorization of $\pi$ is unique, i.e., the parameters  $z(\pi)$, $m^{(v)}$ for $1\leq v\leq z(\pi)$, $h^{(v)}_{j}$, and $\imath_{j}^{(v)}$ for $1\leq j\leq m^{(v)}$, are uniquely determined by the standard $OGS$ canonical form of $\pi$, such that:
    \begin{itemize}
    \item For every $h^{(v)}_{j}$ there exists exactly one $k_{j'}$ (where, $1\leq j'\leq m$), such that $h^{(v)}_{j}=k_{j'}$;
    \item If $h^{(v)}_{j}=k_{j'}$, for some $1\leq v\leq z(\pi)$, ~$1<j<m^{(v)}$, and $1\leq j'\leq m$, then $\imath_{j}^{(v)}=i_{k_{j'}}$;
    \item If $h^{(v_{1})}_{j_{1}}=h^{(v_{2})}_{j_{2}}$, where $1\leq v_{1}<v_{2}\leq z(\pi)$, ~$1\leq j_{1}\leq m^{(v_{1})}$, and  \\ $1\leq j_{2}\leq m^{(v_{2})}$, then necessarily $v_{1}=v_{2}-1$, ~$j_{1}=m^{(v_{1})}$, ~$j_{2}=1$, and $$h^{(v_{2}-1)}_{m^{(v_{2}-1)}}=h^{(v_{2})}_{1}=maj\left(\pi_{(v_{2})}\right)=k_{j'},$$
        for some $j'$, such that $\imath_{m^{(v_{2}-1)}}^{(v_{2}-1)}+\imath_{1}^{(v_{2})}=i_{k_{j'}}$;
    \end{itemize}

\item $$\ell(s_r\cdot \pi) = \begin{cases} \ell(\pi)-1 & r=\sum_{j=1}^{m^{(v)}}\imath_{j}^{(v)} ~~for ~~ 1\leq v\leq z(\pi) \\
\ell(\pi)+1 & otherwise \end{cases}.$$
i.e., $$des\left(\pi\right)=\bigcup_{v=1}^{z(\pi)}des\left(\pi^{(v)}\right)=\{maj\left(\pi^{(v)}\right)~|~1\leq v\leq z(\pi)\};$$
\item
\begin{align*}
\ell(\pi) &= \sum_{v=1}^{z(\pi)}\ell(\pi^{(v)}) = \sum_{v=1}^{z(\pi)}\sum_{j=1}^{m^{(v)}}h^{(v)}_{j}\cdot \imath_{j}^{(v)}-\sum_{v=1}^{z(\pi)}\left(maj\left(\pi^{(v)}\right)\right)^{2} \\ &= \sum_{x=1}^{m}k_{x}\cdot i_{k_{x}}-\sum_{v=1}^{z(\pi)}\left(maj\left(\pi^{(v)}\right)\right)^{2} \\ &= \sum_{x=1}^{m}k_{x}\cdot i_{k_{x}}-\sum_{v=1}^{z(\pi)}{{\left(c^{(v)}\right)}}^{2}, ~~ where ~~ c^{(v)}\in des\left(\pi\right).
\end{align*}
\end{itemize}
\end{theorem}

The paper is organized as follows: In Section \ref{gen-stand-ogs-dn}, we recall some important definitions and basic properties concerning the group $D_{n}$, Then we generalize the definition of the standard $OGS$ and we show the arising exchange laws. In Section \ref{two-para-sub}, we focus on the generalized standard $OGS$ presentation of the elements in two interesting parabolic subgroups of $D_n$ which are isomorphic to $S_n$. In Section \ref{len-of-Dn}, we recall the definition of a normal form of element's of $D_n$, then we give an algorithm to find the Coxeter length of elements of $D_n$ by using the generalized standard $OGS$.

\section{Generalization of the standard OGS for the Coxeter group $D_{n}$}\label{gen-stand-ogs-dn}

First, we recall some important definitions and basic properties concerning the group $D_{n}$, as it defined in every book about Coxeter groups (for example \cite{BB})

\begin{definition}\label{def-dn}
Let $D_{n}$ be the Coxeter group with $n$ generators, with the following presentation:
$$
\left\langle s_{1^{\prime}}, s_{1}, \ldots, s_{n-1}\right| s_{i}^{2}=1, 
\left(s_{i} \cdot s_{i+1}\right)^{3}=1 \text { for } 1 \leq i \leq n-1, 
\left(s_{i} \cdot s_{j}\right)^{2}=1 \text { for }|i-j| \geq 2 ,$$
$$\left (s_{1^{\prime}} \cdot s_{2}\right)^{3}=1, \quad \left (s_{1^{\prime}} \cdot s_{i})^{2}=1\text { for } 1 \leq i \leq n-1 \text { and } i\neq 2,   \right\rangle$$
\end{definition}

\textbf{Basic properties of $D_n$}
\\
\begin{itemize}
\item The group $D_{n}$ is a subgroup of $B_n$,  therefore, it can be presented as a permutation group of the set $[\pm n]=\{\pm i \mid 1 \leq i \leq n\}]$ in such the following holds :  
\begin{itemize}
    \item $\pi(-i)=-\pi(i) \text { for every } i \in[\pm n] . $
    \item  $\{ i |i>0 \quad \text{and} \quad  \pi(i) < 0\} \quad \text{is even.}$
    \end{itemize}
\item  The group $D_{n}$, can be considered as even signed permutation group, where as a subgroup of $B_n$, every element 
$\pi\in D_n$ is uniquely determined by  $\pi(i)$ for $1\leq i\leq n$.
\item An even signed permutation $\pi\in D_n$ is denoted by 
$$
[\pi(1); ~\pi(2); \ldots; ~\pi(n)]
$$ 
 where the number of negative elements of the form $\pi(i)$, where $1\leq i\leq n$ is even. (e.g., $\pi=
[2; -4; 1; -3]
$ is a permutation in $D_{4}$, since for $i=1,2,3,4$, there are 2 (even number) elements of the form  $\pi(i)$  such that  $\pi(i)<0$, namely: $\pi(2)=-4$, $\pi(4)=-3$).
\item As in $S_n$, for $1\leq i\leq n-1$, the Coxeter generators $s_i$ can be considered the permutation which exchanges the element $i$ with the element $i+1$ where in addition exchanges $-i$ with $-(i+1)$ as well, and  $s_{{1}^{\prime}}$ is the permutation which exchanges the element $1$ with the element $-2$ and the element $2$ with the element $-1$.
\item 
$D_n\simeq\mathbb{Z}_2^{n-1}\rtimes S_n$.

\item $|D_n|=2^{n-1}\cdot n!$.
\end{itemize}
Now, we recall the definition of a normal form of elements of $D_{n}$ in Coxeter generators as it defined in \cite{BB} Chapter 3.4.

\begin{definition}\label{D_n-normal-form}

Every element $\pi \in D_{n}$ can be presented uniquely in the following normal form :
 $$\text{norm}(\pi)= \prod_{u=1}^{n-1} \text{norm}_{u}(\pi).$$
 where,
 $$\text{norm}_{u}(\pi) \in \prod_{r=0}^{y_{u}-1} s_{u-r} \quad \text{or} \quad \text{norm}_{u}(\pi) \in \prod_{r=0}^{u-2}s_{u-r} \cdot s_{1^{\prime}} \cdot \prod_{r=1}^{y_{u}^{\prime}} s_{r}.$$
 where:
 $$0 \leqslant y_{u} \leqslant u, \quad  0 \leqslant y_{u}^{\prime} \leqslant u.$$
 \end{definition}
\begin{example}
Let 
$$\pi=
[-2; ~-6; ~-5; ~4; ~-3; ~-7; ~8; ~-1].$$
Then,
$$\text{norm}(\pi) = s_{2} \cdot (s_{3} \cdot s_{2} \cdot s_{1} \cdot s_{1^{\prime}}) \cdot (s_{4} \cdot s_{3} \cdot s_{2}) \cdot (s_{5} \cdot s_{4} \cdot s_{3} \cdot s_{2} \cdot s_{1} \cdot s_{1^{\prime}} \cdot s_{2} \cdot s_{3} \cdot s_{4} \cdot s_{5})\\
\cdot (s_{7} \cdot s_{6} \cdot s_{5} \cdot s_{4} \cdot s_{3} \cdot s_{2} \cdot s_{1^{\prime}}). $$
Hence,

$$\text{norm}_{1}(\pi) = \text{norm}_{6}(\pi) = 1\quad \text{norm}_{2}(\pi)= s_{2}\quad \text{norm}_{3}(\pi) =  s_{3} \cdot s_{2} \cdot s_{1} \cdot s_{1^{\prime}}$$
$$\text{norm}_{4}(\pi) =  s_{4} \cdot s_{3} \cdot s_{2}\quad \text{norm}_{5}(\pi)= s_{5} \cdot s_{4} \cdot s_{3} \cdot s_{2} \cdot s_{1} \cdot s_{1^{\prime}} \cdot s_{2} \cdot s_{3} \cdot s_{4} \cdot s_{5}$$ $$\text{norm}_{7}(\pi)=s_{7} \cdot s_{6} \cdot s_{5} \cdot s_{4} \cdot s_{3} \cdot s_{2} \cdot s_{1^{\prime}}.$$

\end{example}

\subsection{The generalized standard OGS for $D_{n}$}

In this subsection we define the generalized standard $OGS$ for $D_n$. First, notice, the set of  generalized standard $OGS$ contains the standard $OGS$ for $S_n$ which is defined in \cite{S1}. \\

Therefore, first we recall the definition of $t_i$ as it defined in \cite{S1}.\\
\begin{definition}\cite{S1}
for $i= 2,3, \ldots,n$ let $t_{i}$ to be :
$$t_{i} = \prod_{j=1}^{n-1} s_{j}.$$
\end{definition}

\begin{remark}\cite{S1}
For every $2 \leqslant i \leqslant n$, the permutation presentation of $t_{i}$ satisfies the following properties:
$$\begin{array}{l}
t_{i}(1)=i \\
t_{i}(j)=j-1, \text { for } 2 \leq j \leq i \\
t_{i}(j)=j, \text { for } i+1 \leq j \leq n
\end{array}$$
\end{remark}

\begin{definition}\label{def w}
For $k=1,2,\ldots,n$ let $w_{k}$ to be :
$$w_{k} = s_{k}\cdot s_{k-1} \cdots s_{1}\cdot s_{1^{\prime}} \cdot s_{2} \cdot s_{3} \cdots s_{k}$$
\end{definition}

\begin{lemma}\label{order-w}
For every $1 \leqslant k \leqslant n-1$,
$$w_{k}^{2}=1.$$
\end{lemma}

\begin{proof}

$$w_{k}^{2} = $$ $$= (s_{k} \cdot s_{k-1} \cdots s_{1}\cdot s_{1^{\prime}} \cdot  s_{2} \cdots s_{k-1} \cdot s_{k}) (\cdot s_{k} \cdot s_{k-1} \cdots s_{1}\cdot s_{1^{\prime}} \cdot  s_{2} \cdots s_{k-1} \cdot s_{k})=$$ $$=s_1\cdot s_{1^\prime}\cdot s_1\cdot s_{1^\prime}=s_1\cdot s_{1^\prime}\cdot s_{1^\prime}\cdot s_1=1.$$
\end{proof}

\begin{remark}\label{perm-w}
For every $1 \leqslant k \leqslant n$, the permutation presentation of $w_{k}$ satisfies the following properties: 
$$\begin{array}{l}
w_{k}(1)=-1 \\
w_{k}(k+1)=-(k+1) \\
w_{k}(j)=j, \text { for } 2\leqslant j \leqslant k \quad \text{and} \quad k+2\leqslant j \leqslant n.
\end{array}$$
\end{remark}

\begin{theorem}\label{ogs-dn}
 Every $ \pi \in D_{n}$ has a unique presentation in the form :
$$\pi = w_{1}^{j_{1}} \cdot t_{2}^{i_{2}} \cdot w_{2}^{j_{2}} \cdot t_{3}^{i_{3}} \cdots w_{n-1}^{j_{n-1}} \cdot t_{n}^{i_{n}}$$

Where $$0 \leqslant i_{k} \leqslant k-1 \quad \text{and} \quad 0 \leqslant j_{k} \leqslant 1$$
\end{theorem}

\begin{proof}
The proof is by induction on $n$. for $n=2$, it is easy to see that $D_2$ is generated by $w_1$ and  $t_2$.  Since $D_2$ is a non-cyclic  abelian  group of order $4$, every element of $D_2$ has a unique presentation in the form $w_1^{j_1}\cdot t_2^{i_2}$, where $j_1, i_2\in\{0,1\}$. Now, assume by induction that theorem holds for every $k$ such that $k\leq n-1$, i.e., every element of $D_k$ has a unique presentation in the form
$$w_1^{j_1}\cdot t_2^{i_2}\cdot w_2^{j_2}\cdot t_3^{i_3}\cdots w_{k-1}^{j_{k-1}}\cdot t_{k}^{i_{k}}.$$ 
Denote by $\dot{D}_{n-1}$ the parabolic subgroup of $D_n$ generated by $s_1^{\prime}, s_1, \ldots s_{n-2}$. Easy to see that $\dot{D}_{n-1}$ is isomorphic to $D_{n-1}$ (which satisfy the theorem by the induction hypothesis). Notice, that in the permutation presentation of $D_n$, every element $x\in \dot{D}_{n-1}$ satisfies $x(n)=n$. Now, look at the right cosets of $\dot{D}_{n-1}$ in $D_n$. There are $2n$ different right cosets, where every two elements $x$ and $y$ in the same right coset of $\dot{D}_{n-1}$ in $D_n$ satisfy $x(n)=y(n)$. 
Now, notice that for $0\leq i_n<n$, the elements of the form $t_n^{i_n}$ and $w_{n-1}\cdot t_n^{i_n}$ satisfy the following properties:

 $$t_{n}^{i_n}(n)=n-i_n\quad \quad w_{n-1}\cdot t_{n}^{i_n}(n)=-(n-i_n).$$

Hence, for $0\leq i_n\leq n-1$, $0\leq j_{n-1}\leq 1$, the elements of the form  $w_{n-1}^{j_{n-1}}\cdot t_{n}^{i_n}$  gives the $2n$ different images of $n$ in the permutation presentation. Hence, for $0\leq i_n\leq n-1$, $0\leq j_{n-1}\leq 1$, we have the following $2n$ different right cosets of $\dot{D}_{n-1}$ in $D_n$ :
$$\dot{D}_{n-1}w_{n-1}^{j_{n-1}}\cdot t_{n}^{i_n}.$$
Then the result of the theorem holds for $k=n$.
\end{proof}

\begin{example}
Consider the element $\pi=
[-2; ~-1; ~-4; ~-3]
$ of $D_4$. Now, we construct the $OGS$ presentation of the element, as it described in Theorem \ref{ogs-dn}.
First, notice $\pi(4)=-3$. Since, $-3<0$ we have $j_3=1$, and hence $t_4^{i_4}(-4)=-3$, which applies  $i_4=4-3=1$. Thus we conclude:
$$\pi\in \dot{D}_3 ~w_3\cdot t_4.$$
Now, look at $\pi(3)=-4$. Notice, that $[w_3\cdot t_4](1)=-4$. Hence, we look for  $j_2$ and $i_3$, such that $[w_2^{j_2}\cdot t_3^{i_3}](3)=1$. Since $1>0$, we have $j_2=0$, and hence $t_3^{i_3}(3)=1$,  which applies $i_3=3-1=2$. 
Thus we conclude: 
$$\pi\in \dot{D}_2 ~t_3^2\cdot w_3\cdot t_4.$$
Continuing by the same process, looking at $\pi(2)=-1$. 
Notice, that \\ $[t_3^2\cdot w_3\cdot t_4](-1)=-1$. Hence, we look for $j_1$ and $i_2$, such that $[w_1^{j_1}\cdot t_2^{i_2}](2)=-1$. Since $-1<0$, we have $j_1=1$, and hence $t_2^{i_2}(2)=1$, which applies $i_2=2-1=1$. 
Thus we conclude: 
$$\pi = w_1\cdot t_2\cdot t_3^2\cdot w_3\cdot t_4.$$

\end{example}

\begin{remark}
We call the presentation of elements of $D_n$ which has been  shown in Theorem \ref{ogs-dn}  the generalized standard $OGS$ presentation of $D_n$.
\end{remark}

We show by Theorem \ref{bk-in-dn} a presentation of a subgroup of $D_n$ which is isomorphic to $B_m$ for every $m<n$, by using the generalized standard $OGS$ of $D_n$.

\begin{theorem}\label{bk-in-dn}
 The elements  $w_1^{j_1}\cdot t_{2}^{i_2}\cdots t_{m}^{i_m}\cdot w_{m}^{j_m}$, such that $m$ is a positive integer less than $n$, ~$0\leq j_k\leq 1$ and $0\leq i_k\leq k-1$ form a subgroup of $D_n$ which is isomorphic to $B_{m}$.
\end{theorem}

\begin{proof}
The proof is by induction on $m$. For $m=1$, the group generated by $w_1$ is cyclic of order $2$, which is isomorphic to $B_1$. Notice, by Theorem \ref{ogs-dn}, the subgroup which is generated by $w_1, t_2, w_2, \ldots, w_{m-1}, t_m$, denoted by $\dot{D}_m$ is isomorphic to $D_m$, where the permutation presentation of every element $x\in \dot{D}_m$ satisfies the following properties:
\begin{itemize}
    \item $x(k)=k$ for every $m+1\leq k\leq n$;
    \item $|\{1\leq k\leq m | x(k)<0\}|$ is even. 
\end{itemize}
Notice, the permutation presentation $w_m^{j_m}$ where $j_m=0$ or $j_m=1$
satisfies: 
\begin{itemize}
    \item Since $w_m^0$ is the identity, $w_m^0(k)=k$ for every $1\leq k\leq n$;
    \item $w_m^1(1)=-1\quad w_m^1(m+1)=-(m+1)\quad w_m^1(k)=k$ for $2\leq k\leq  n$, where  $k\neq m+1$.
\end{itemize}
Therefore, we have that the permutation presentation of every element of the form $y=x\cdot w_m^{j_m}$, where $x\in \dot{D}_m$ satisfies $y(k)=k$ for $m+2\leq k\leq n$, and either $y(m+1)=m+1$ or $y(m+1)=-(m+1)$. 
Since $y\in \dot{D}_{m+1}$, $|\{1\leq k\leq m+1 | y(k)<0\}|$ is even. Hence, we conclude the following properties:
\begin{itemize}
    \item If $y(m+1)=m+1$ (i.e., $j_m=0$), $|\{1\leq k\leq m | y(k)<0\}|$ is even;
    \item If $y(m+1)=-(m+1)$ (i.e., $j_m=1$), $|\{1\leq k\leq m | y(k)<0\}|$ is odd.
\end{itemize}
Hence, the permutation presentation of the elements of the form \\ $w_1^{j_1}\cdot t_{2}^{i_2}\cdots t_{m}^{i_m}\cdot w_{m}^{j_m}$ reduced to the set $[\pm m]$, such that $m$ is a positive integer less than $n$, ~$0\leq j_k\leq 1$ and $0\leq i_k\leq k-1$ for every $1\leq k\leq m$, contains once every sign permutation of $[\pm m]$. Hence, the subgroup is isomorphic to $B_m$.
\end{proof}

Now, we define a natural homomorphism from $D_n$ to $S_n$ by identifying $s_1^{\prime}$ and $s_1$, and then we see the connection of the generalized standard $OGS$ of $D_n$ to the standard $OGS$ of the homomorphic image $S_n$.

\begin{definition}\label{hom-dn-sn}
Consider the natural homomorphism $\Phi$ from $D_{n}$ to $S_{n}$, such that
\begin{itemize}

    \item $\Phi(s_{1^{\prime}})=s_{1}$
    \item $\Phi(s_{i})= s_{i} , \quad \text{where} \quad i \neq 1^{\prime}$
\end{itemize}
Then denote $$\Phi(\pi) := \pi^{\prime}.$$

Notice, $\pi^{\prime}$ is the permutation of $S_n$ which satisfies:
$$\pi^{\prime}(j)=|\pi(j)|, \text { for } 1\leqslant j \leqslant n$$
\end{definition}

\begin{corollary}\label{w-reduce}
  Let $\pi\in D_n$ be presented by the generalized standard $OGS$ as it is described in Theorem \ref{ogs-dn}:
   
   $$\pi = w_{1}^{j_{1}} \cdot t_{2}^{i_{2}} \cdot w_{2}^{j_{2}} \cdot t_{3}^{i_{3}} \cdots w_{n-1}^{j_{n-1}} \cdot t_{n}^{i_{n}}$$
Then, $$\pi^{\prime} = t_{2}^{i_{2}}\cdots t_{n}^{i_{n}}.$$
\end{corollary}

\begin{proof}
by Definition \ref{def w}, $w_k=s_k\cdot s_{k-1}\cdots s_2\cdot s_1\cdot s_1^{\prime}\cdot s_{2}\cdots s_k$, for every $1\leq k\leq n-1$. Then, by Definition \ref{hom-dn-sn}, $$\Phi(w_k)=\Phi(s_k\cdot s_{k-1}\cdots s_2\cdot s_1\cdot s_1^{\prime}\cdot s_{2}\cdots s_k)=s_k\cdot s_{k-1}\cdots s_2\cdot s_1\cdot s_1\cdot s_{2}\cdots s_k=1.$$
Hence, in case 
$$\pi = w_{1}^{j_{1}} \cdot t_{2}^{i_{2}} \cdot w_{2}^{j_{2}} \cdot t_{3}^{i_{3}} \cdots w_{n-1}^{j_{n-1}} \cdot t_{n}^{i_{n}}$$
by Definition \ref{hom-dn-sn}, we get that
$$\pi^{\prime}=\Phi(\pi)= w_{1}^{j_{1}} \cdot t_{2}^{i_{2}} \cdot w_{2}^{j_{2}} \cdot t_{3}^{i_{3}} \cdots w_{n-1}^{j_{n-1}} \cdot t_{n}^{i_{n}}=t_{2}^{i_{2}}\cdot t_3^{i_3}\cdots t_{n}^{i_{n}}. $$
\end{proof}

Now we show the exchange laws which arise from the generalized standard $OGS$ of $D_n$.\\

\begin{proposition}\label{exchange-dn}
Consider the elements of the group $D_n$ expressed by the generalized standard $OGS$ presentation. Assume $q>p$, then the following exchange laws holds:

First, we look at the exchange law of $t_{q}^{i_{q}}\cdot t_{p}^{i_p}$ as it defined in \cite{S1} ;\\

     Now, we show the exchange laws for products which involves generator of a form $w_k$.

\begin{itemize}
    \item  $w_{q}\cdot w_{p} = w_{p} \cdot w_{q}.$ 
     \item $t_{q}^{i_{q}} \cdot w_{p}=\left\{\begin{array}{ll} w_{i_{q}} \cdot t_{q}^{i_{q}} & q= p + i_{q} \\
     w_{i_{q}} \cdot w_{p+i_{q}} \cdot t_{q}^{i_{q}} & q > p + i_{q} \\
     w_{p+i_{q}-q} \cdot w_{i_{q}} \cdot t_{q}^{i_{q}} & q < p + i_{q}
     \end{array}\right.
     $
    \item $w_{q} \cdot t_{p}^{i_{p}} = w_{i_{p}} \cdot t_{p}^{i_{p}} \cdot w_{q}.$ \quad where \quad $p<q$
    \end{itemize}
    \end{proposition}
\begin{proof}
\begin{itemize}
    \item First, consider the exchange law for $w_{q}\cdot w_{p}$. \\
    
    By looking at the permutation presentation of $w_p$ and $w_q$ we get :
    $$ w_{q}=
   [-1; ~2; ~3;\ldots; ~p; ~p+1; \ldots; ~q; ~-(q+1); ~q+2; ~q+3;\ldots; ~n].
    $$
   
    $$
    w_{p}=
   [-1; ~2; ~3; \ldots; ~p; ~-(p+1); ~p+2;\ldots; ~q+1; \ldots; ~n].
    $$
    Hence we get:
    $$
    w_{q}\cdot w_{p}=[-1; ~2; ~3;\ldots; ~p; ~-(p+1); ~p+2 \ldots; ~q; ~-(q+1); ~q+2; ~q+3;\ldots; ~n].
    $$
    Then obviously:
    $$w_{q}\cdot w_{p} = w_{p} \cdot w_{q}.$$
   \item  Now, consider the exchange law for $t_{q}^{i_{q}} \cdot w_{p}.$
   
   By looking at the permutation presentation of  $w_{p}$:$$1 \rightarrow -1,\quad \quad  p+1 \rightarrow -(p+1)$$
   $$j \rightarrow j \quad \text{for} \quad j\neq 1 \quad \text{or} \quad j\neq p+1.$$
   and by looking at the permutation presentation of $t_{q}^{i_{q}}$:
   $$i_{q}+1\rightarrow 1$$ 
   Then we get the following three subcases:
   \begin{itemize}
       \item $p+i_{q} < q$;
       \item $p+i_{q} > q$;
       \item $p+i_{q}=q$.
   \end{itemize}
   The case $p+i_{q} < q$: $$p+i_{q}+1 \leqslant q$$ 
   $$p+i_{q}+1 \rightarrow p+1.$$
   
   The case $p+i_{q} > q$:
   $$p+i_{q} - q+1 \rightarrow p+1.$$
   
   The  case  $p+i_{q}=q$:
   $$i_{q}+1 \rightarrow 1 \rightarrow -1 \rightarrow -(i_{q}+1)$$
   $$1\rightarrow p+1 \rightarrow -(p+1) \rightarrow -1.$$
   \item First, consider the permutation presentation of $w_{q}$ and $t_{p}^{i_{p}}$
   $$
   w_{q}=
   [-1; ~2; ~3; \ldots; ~q; ~-(q+1); ~q+2; ~q+3;\ldots; ~n]
  $$
  $$
  t_p^{i_p}=[p+1-i_p; ~p+2-i_p; \ldots; ~p; ~1; ~2;\ldots; ~q;  ~q+1;\ldots  ~n].
  $$
  
  $$
  w_{q} \cdot t_{p}^{i_{p}}=[-(p+1-i_p); ~p+2-i_p; \ldots; ~p; ~1; ~2;\ldots; ~q;  ~-(q+1); ~q+2; \ldots  ~n]
   $$
   $$w_{q} \cdot t_{p}^{i_{p}} = t_{p}^{i_{p}} \cdot w_{p-i_{p}} \cdot w_{q}$$ 
   Now, from the previous exchange laws, we have :
    $$w_{q} \cdot t_{p}^{i_{p}} = w_{i_{p}} \cdot t_{p}^{i_{p}} \cdot w_{q}.$$

\end{itemize}

\end{proof}

In Theorem \ref{ogs-dn}, we introduce a generalized standard $OGS$ for the group $D_n$, which involves two types of generators $t_k$ for $2\leq k\leq n$, and $w_L$ \\ for $1\leq L\leq n-1$. In Definition \ref{hom-dn-sn}, there was introduced a natural homomorphism  $\Phi$ from $D_n$ onto $S_n$, where by Corollary \ref{w-reduce},  the presentation of $\pi^{\prime}=\Phi(\pi)$ is derived from the presentation of $\pi$ by omitting the generators of type $w_L$ from the generalized standard $OGS$ presentation of $\pi$. Hence, it motivates us to find other properties of the group $D_n$ by considering the elements of the group presented in different forms by using $t_k$ and $w_L$. We show by the next theorem that all the elements of $D_n$ which can be expressed as product of powers of $t_k$ (i.e., $t_{k_{1}}^{i_{k_1}}\cdot t_{k_2}^{i_{k_2}}\cdots t_{k_m}^{i_{k_m}}$) or as a product of elements of the form  $w_L$ for $1\leq L\leq n-1$ (i.e., $w_{L_1}\cdot w_{L_2}\cdots w_{L_m}$)  form two important subgroups of $D_n$. We show that the mentioned two subgroups are strongly connected to the homomorphism $\Phi$ from $D_n$ onto $S_n$.
 \begin{theorem}\label{subgroup-t-w}
Consider the group $D_n$, with the generalized standard $OGS$ presentation as it is presented in Theorem \ref{ogs-dn}. 
$$\pi = w_{1}^{j_{1}} \cdot t_{2}^{i_{2}} \cdot w_{2}^{j_{2}} \cdot t_{3}^{i_{3}} \cdots w_{n-1}^{j_{n-1}} \cdot t_{n}^{i_{n}}$$

where $$0 \leqslant i_{k} \leqslant k-1 \quad \text{and} \quad 0 \leqslant j_{k} \leqslant 1$$
Then the following holds:
\begin{itemize}
    \item All the elements such that $j_{k}=0$ for every $1\leq k\leq n-1$ form a subgroup which is isomorphic to $S_n$, and contains all  elements $\pi\in D_n$ such that the generalized standard $OGS$ presentation of $\pi$, as it described in Theorem \ref{ogs-dn}, is the same to the standard $OGS$ presentation of  $\pi^{\prime}=\pi$, as it is described in Theorem \ref{canonical-sn}. (i.e.,The parabolic subgroup of $D_n$, which is generated by $\{s_1, s_2, \ldots, s_{n-1}\}$). We denoted the mentioned subgroup by $S^{\circ}_n$. 
    \item All the elements such that $i_{k}=0$ for every $2\leq k\leq n$ form a subgroup which is isomorphic to $\mathbb{Z}_2^{n-1}$,  and contains all elements $\pi\in D_n$ such that  $\pi^{\prime}=1$. We denote the mentioned subgroup by $Id^{\bullet}_n$.
    
\end{itemize}
  
\end{theorem}
\begin{proof}
Consider the generalized standard $OGS$ presentation of $D_n$, where every element $\pi\in D_n$ is presented in the following form: 
$$\pi = w_{1}^{j_{1}} \cdot t_{2}^{i_{2}} \cdot w_{2}^{j_{2}} \cdot t_{3}^{i_{3}} \cdots w_{n-1}^{j_{n-1}} \cdot t_{n}^{i_{n}}$$

where $$0 \leqslant i_{k} \leqslant k-1 \quad \text{and} \quad 0 \leqslant j_{k} \leqslant 1.$$

The elements of $D_n$ which we get by considering $j_k=0$ for $1\leq k\leq n-1$, are of the form
\begin{equation}\label{pure-t}
t_2^{i_2}\cdot t_3^{i_3}\cdots t_n^{i_n}
\end{equation}
where, $$0 \leqslant i_{k} \leqslant k-1\quad \text{for}\quad 2\leq k\leq n.$$
Hence, by Theorem \ref{canonical-sn}, the set of all elements of $D_n$ which  generalized standard $OGS$ presentation can be expressed by Equation \ref{pure-t} form a subgroup of $D_n$ which is isomorphic to $S_n$. Now, by Theorem \ref{canonical-sn}, $t_k=\prod_{r=1}^{k-1} s_r$,  for $2\leq k\leq n$. Hence, there is no occurrence of the Coxeter generator $s_1^{\prime}$ in the normal form of $t_k$ for $2\leq k \leq n$. Thus by Corollary \ref{w-reduce}, the generalized standard $OGS$ presentation of $\pi$ is the same to the standard $OGS$ presentation of $\pi^{\prime}$, and by considering $\pi$ as an element of $S_n$, $\pi$ presents the same permutation as $\pi^{\prime}$.\\

Moreover, If there is $1\leq k\leq n-1$ such that $j_k\neq 0$ in the generalized standard $OGS$ presentation of an element $\pi\in D_n$, then by  Corollary \ref{w-reduce},  the standard $OGS$ presentation of $\pi^{\prime}$ (as it is described in Theorem \ref{canonical-sn}) is different from the generalized standard $OGS$ presentation of $\pi$ (as it is described in Theorem \ref{ogs-dn}), since there is an appearance of $w_k$ in the generalized standard $OGS$ presentation of an element $\pi\in D_n$, which does  not occur in the standard $OGS$ presentation of $\pi^{\prime}$.   \\ 

Now, consider the elements of $D_n$ such that $i_k=0$ for $2\leq n$ in generalized standard $OGS$ presentation. Then, we get the elements of the form:
\begin{equation}\label{pure-w}
w_1^{j_1}\cdot w_2^{j_2}\cdots w_{n-1}^{j_{n-1}} 
\end{equation}
where, $$0 \leqslant j_{k} \leqslant 1\quad \text{for}\quad 1\leq k\leq n-1.$$
Notice, by Proposition \ref{exchange-dn}, $w_q\cdot w_p=w_p\cdot w_q$ for every $1\leq p, q\leq n-1$, and by Lemma \ref{order-w}, $w_k^2=1$ for every $1\leq k\leq n-1$. Therefore, the elements of $D_n$ which generalized standard $OGS$ presentation can be expressed by Equation \ref{pure-w}  form an abelian group which is isomorphic to $\mathbb{Z}_2^{n-1}.$
Now, by Corollary \ref{w-reduce}, every element $\pi\in D_n$ such that the generalized standard $OGS$ presentation of it can be expressed by Equation \ref{pure-w} satisfies $\pi^{\prime}=1$.
Now, we prove that the generalized standard $OGS$ presentation of an element $\pi$ such that $\pi^{\prime}=1$ can be expressed by Equation \ref{pure-w}. We prove it by the permutation presentation of $\pi$. By Definition \ref{hom-dn-sn}, if $\pi^{\prime}=1$, then $|\pi(k)|=k$ for every $1\leq k\leq n$. Hence, either $\pi(k)=k$ or $\pi(k)=-k$ for every $1\leq k\leq n$. Let $k_1, k_2, \ldots k_q$ be $q$ integers such that
\begin{itemize}
    \item $2\leq k_1<k_2<\ldots < k_q\leq n$.
    \item $\pi(k_1)=-k_1, \quad \pi(k_2)=-k_2,\quad \ldots,\quad \pi(k_q)=-k_q$.
    \item If $k\neq k_p$ for any $1\leq p\leq q$ and $2\leq k\leq n$, then $\pi(k)=k$.
\end{itemize}
Now, since  by Definition \ref{def-dn}, every $\pi$ such that $\pi^{\prime}=1$, satisfies that the number of $k$ such that $\pi(k)=-k$ for $1\leq k\leq n$ is even, $\pi(1)$ is determined by the parity of $q$ (i.e, if $q$ is even then $\pi(1)=1$ and if $q$ is odd then $\pi(1)=-1$.
By Remark \ref{perm-w}, for every $1\leq k\leq n-1$, $w_k$ satisfies the following  $$\begin{array}{l}
w_{k}(1)=-1 \\
w_{k}(k+1)=-(k+1) \\
w_{k}(j)=j, \text { for } 2\leqslant j \leqslant k \quad \text{and} \quad k+2\leqslant j \leqslant n.
\end{array}.$$
Thus we get that 
$$\pi=w_{k_1-1}\cdot w_{k_2-1}\cdots w_{k_q-1}.$$
Hence, every $\pi\in D_n$ such that $\pi^{\prime}=1$,  the generalized standard $OGS$ presentation of  $\pi$ can be expressed by Equation \ref{pure-w}.

\end{proof}

\begin{corollary}\label{twtw}
Let $\pi\in D_n$, then $\pi$ can be presented in the form:
$$\pi=\pi_{1}^{\circ}\cdot \pi_{1}^{\bullet}\cdot \pi_{2}^{\circ}\cdot \pi_{2}^{\bullet}\cdots \pi_{\mu-1}^{\circ}\cdot \pi_{\mu-1}^{\bullet}\cdot \pi_{\mu}^{\circ}.$$
  such that the following holds:
  
  \begin{itemize}
            \item  For $1\leq j\leq \mu$, ~$\pi_{j}^{\circ}\in S^{\circ}_n$ (as it is defined in Theorem \ref{subgroup-t-w}) such that :
            \begin{itemize}
                \item For $j=1$ or $j=\mu$, either $\pi_{j}^{\circ}=1$ or the generalized standard $OGS$ presentation of $\pi_{j}^{\circ}$ has the following form $$\pi_{j}^{\circ}=t_{k_{r_{j-1}+1}}^{i_{k_{r_{j-1}+1}}}\cdots t_{k_{r_j}}^{i_{k_{r_j}}}$$ for some integer $r_j$;
                 \item For $2\leq j\leq \mu-1$, the generalized standard $OGS$ presentation of $\pi_{j}^{\circ}$ has the following form: $$\pi_{j}^{\circ}=t_{k_{r_{j-1}+1}}^{i_{k_{r_{j-1}+1}}}\cdots t_{k_{r_j}}^{i_{k_{r_j}}}$$ for some integer $r_j$,
            \end{itemize}
            \item For $1\leq j\leq \mu-1$, ~$\pi_{j}^{\bullet}\in Id^{\bullet}_n$ (as it is defined in Theorem \ref{subgroup-t-w}) such that:
            \begin{itemize}
                \item The generalized standard $OGS$ presentation of $\pi_{j}^{\bullet}$ has the following form:
            $$\pi_{j}^{\bullet}=w_{L_{j_1}}\cdot w_{L_{j_2}}\cdots w_{L_{j_{\nu_j}}},$$ where $\nu_j$ denoted the number of non-zero powers of elements of the form $w_{L_j}$ in the generalized standard $OGS$ presentation (as it is described in Theorem \ref{ogs-dn}) of $\pi_j^{\bullet}$, ~ $L_{j_1}\geq k_{r_j}$,  ~$L_{j_u}>L_{j_{u-1}}$ for every $2\leq u\leq \nu_j$, and $L_{j_{\nu_j}}<t_{k_{r_j}+1}$.
            \end{itemize}
            
            \end{itemize}

  Then, the generalized standard $OGS$ presentation of $\pi$ is presented as follow:
  \begin{itemize}
      \item In case $\pi_1^{\circ}\neq 1$ and $\pi_{\mu}^{\circ}\neq 1$:
  
  $$\pi= t_{k_{1}}^{i_{k_{1}}} \cdots t_{k_{r_{1}}}^{i_{k_{r_{1}}}} \cdot w_{L_{1_1}} \cdots w_{L_{1_{\nu_{1}}}} \cdot t_{k_{r_{1}+1}}^{i_{k_{r_{1}+1}}} \cdots t_{k_{r_{2}}}^{i_{k_{r_{2}}}}\cdot w_{L_{2_1}} \cdots w_{L_{2_{\nu_{2}}}} \cdot t_{k_{r_{2}+1}}^{i_{k_{r_{2}}+1}} \cdots $$ $$\cdots  t_{k_{r_{\mu-1}}}^{i_{k_{r_{\mu-1}}}} \cdot  w_{L_{{\mu-1}_1}} \cdots w_{L_{{\mu-1}_{\nu_{\mu-1}}}} \cdot t_{k_{r_{\mu-1}+1}}^{i_{k_{r_{\mu-1}+1}}} \cdots t_{k_{r_{\mu}}}^{i_{k_{r_{\mu}}}}.$$

and then,

 $$\pi^{\prime}=\pi_{1}^{\circ}\cdot \pi_{2}^{\circ}\cdots \pi_{\mu}^{\circ}=t_{k_{1}}^{i_{k_{1}}} \cdots t_{k_{r_{1}}}^{i_{k_{r_{1}}}} \cdot t_{k_{r_{1}+1}}^{i_{k_{r_{1}+1}}} \cdots t_{k_{r_{2}}}^{i_{k_{r_{2}}}}\cdots t_{k_{r_{\mu}}}^{i_{k_{r_{\mu}}}}.$$
 
  \item In case $\pi_1^{\circ}\neq 1$ and $\pi_{\mu}^{\circ}=1$:
 
  $$\pi= t_{k_{1}}^{i_{k_{1}}} \cdots t_{k_{r_{1}}}^{i_{k_{r_{1}}}} \cdot w_{L_{1_1}} \cdots w_{L_{1_{\nu_{1}}}} \cdot t_{k_{r_{1}+1}}^{i_{k_{r_{1}+1}}} \cdots t_{k_{r_{2}}}^{i_{k_{r_{2}}}}\cdot w_{L_{2_1}} \cdots w_{L_{2_{\nu_{2}}}} \cdot t_{k_{r_{2}+1}}^{i_{k_{r_{2}}+1}} \cdots $$ $$\cdots  t_{k_{r_{\mu-1}}}^{i_{k_{r_{\mu-1}}}} \cdot  w_{L_{{\mu-1}_1}} \cdots w_{L_{{\mu-1}_{\nu_{\mu-1}}}}.$$

and then,

 $$\pi^{\prime}=\pi_{1}^{\circ}\cdot \pi_{2}^{\circ}\cdots \pi_{\mu}^{\circ}=\pi_{1}^{\circ}\cdot \pi_{2}^{\circ}\cdots \pi_{\mu-1}^{\circ} =t_{k_{1}}^{i_{k_{1}}} \cdots t_{k_{r_{1}}}^{i_{k_{r_{1}}}} \cdot t_{k_{r_{1}+1}}^{i_{k_{r_{1}+1}}} \cdots t_{k_{r_{2}}}^{i_{k_{r_{2}}}}\cdots t_{k_{r_{\mu-1}}}^{i_{k_{r_{\mu-1}}}}.$$

 \item In case $\pi_1^{\circ}=1$ and $\pi_{\mu}^{\circ}\neq 1$ (Notice, in this case $r_1=0$):
  
  $$\pi = w_{L_{1_1}} \cdots  w_{L_{1_{\nu_{1}}}} \cdot t_{k_{r_{1}+1}}^{i_{k_{r_{1}+1}}} \cdots t_{k_{r_{2}}}^{i_{k_{r_{2}}}}\cdot w_{L_{2_1}} \cdots w_{L_{2_{\nu_{2}}}} \cdot t_{k_{r_{2}+1}}^{i_{k_{r_{2}}+1}} \cdots $$ $$\cdots  t_{k_{r_{\mu-1}}}^{i_{k_{r_{\mu-1}}}} \cdot  w_{L_{{\mu-1}_1}} \cdots w_{L_{{\mu-1}_{\nu_{\mu-1}}}} \cdot t_{k_{r_{\mu-1}+1}}^{i_{k_{r_{\mu-1}+1}}} \cdots t_{k_{r_{\mu}}}^{i_{k_{r_{\mu}}}}.$$

and then,

 $$\pi^{\prime}=\pi_{1}^{\circ}\cdot \pi_{2}^{\circ}\cdots \pi_{\mu}^{\circ}=\pi_{2}^{\circ}\cdot \pi_{3}^{\circ}\cdots \pi_{\mu}^{\circ} =t_{k_{1}}^{i_{k_{1}}} \cdots t_{k_{r_{2}}}^{i_{k_{r_{2}}}} \cdot t_{k_{r_{2}+1}}^{i_{k_{r_{2}+1}}} \cdots t_{k_{r_{2}}}^{i_{k_{r_{2}}}}\cdots t_{k_{r_{\mu}}}^{i_{k_{r_{\mu}}}}.$$

 \item In case $\pi_1^{\circ}=1$ and $\pi_{\mu}^{\circ}=1$ (Notice, in this case $r_1=0$):
  
  $$\pi = w_{L_{1_1}} \cdots w_{L_{1_{\nu_{1}}}} \cdot t_{k_{r_{1}+1}}^{i_{k_{r_{1}+1}}} \cdots t_{k_{r_{2}}}^{i_{k_{r_{2}}}}\cdot w_{L_{2_1}} \cdots w_{L_{2_{\nu_{2}}}} \cdot t_{k_{r_{2}+1}}^{i_{k_{r_{2}}+1}} \cdots $$ $$\cdots  t_{k_{r_{\mu-1}}}^{i_{k_{r_{\mu-1}}}} \cdot  w_{L_{{\mu-1}_1}} \cdots w_{L_{{\mu-1}_{\nu_{\mu-1}}}}.$$

and then,

 $$\pi^{\prime}=\pi_{1}^{\circ}\cdot \pi_{2}^{\circ}\cdots \pi_{\mu}^{\circ}=\pi_{2}^{\circ}\cdot \pi_{3}^{\circ}\cdots \pi_{\mu-1}^{\circ} =t_{k_{1}}^{i_{k_{1}}} \cdots t_{k_{r_{2}}}^{i_{k_{r_{2}}}} \cdot t_{k_{r_{2}+1}}^{i_{k_{r_{2}+1}}} \cdots t_{k_{r_{2}}}^{i_{k_{r_{2}}}}\cdots t_{k_{r_{\mu-1}}}^{i_{k_{r_{\mu-1}}}}.$$

\end{itemize}

\end{corollary}

\begin{proof}
The generalized standard $OGS$ of  $\pi$ as follow, 
$$\pi=w_1^{j_1}\cdot t_2^{i_2}\cdot w_2^{j_2}\cdots t_n^{i_n},$$
If $i_k=0$ or $j_k=0$, we may omit  $t_k^{i_k}$ $w_k^{j_k}$ since $t_k^0=w_k^0=1$. Then by considering generators $t_k$ and $w_k$ in the generalized standard $OGS$  with non-zero powers only, we get the result of the proposition.
\end{proof}
In Theorem \ref{ogs-dn} we introduced the generalized standard $OGS$ presentation of the elements in the group $D_n$, by using two types of generators, $t_k$ for \\ $2\leq k\leq n$ and $w_L$ for $1\leq L\leq n-1$. In Theorem \ref{subgroup-t-w}, we show that the elements of $D_n$ which can be presented by using just one of the types of the generators ($t_k$ or $w_L$) in the generalized standard $OGS$ presentation, form a subgroup of $D_n$, where the subgroup which contains only products of elements of the form $w_L$ (for $1\leq L\leq n-1$) is denoted by $Id_n^{\bullet}$ and the subgroup which contains only products of elements of the form $t_k$ (for $2\leq k\leq n$) is denoted by $S_n^{\circ}$.
Now, we show an algorithm  to transform an element $\pi\in D_n$ presented by the generalized standard $OGS$ presentation, to a presentation of $\pi$ as a product of an element in $Id_n^{\bullet}$  and an element in $S_n^{\circ}$, where by Theorem \ref{wt} we give an explicit formula for it in case where $\pi^{\prime}$ is a standard $OGS$ elementary element (as it is defined in Definition \ref{elementary}). We start with the following lemma, which we use in the proof of Theorem \ref{wt}.

\begin{lemma}\label{tw-standard}
Let $\pi=\pi_1\cdot w_{q}$ such that the following holds:
\begin{itemize}
    \item $\pi_1\in S^{\circ}_n$;
    \item $\pi_{1}=t_{k_1}^{i_{k_1}}\cdot t_{k_2}^{i_{k_2}}\cdots t_{k_m}^{i_{k_m}}$ is a standard $OGS$ elementary element (as it is defined in Definition \ref{elementary}) by considering $\pi_1$ as an element of $S_n$;
    \item $q\geq k_m$. i.e., the presentation of 
    $$\pi=t_{k_1}^{i_{k_1}}\cdot t_{k_2}^{i_{k_2}}\cdots t_{k_m}^{i_{k_m}}\cdot w_{q}$$ is a generalized standard $OGS$ presentation of $\pi\in D_n$ as it is presented in Theorem \ref{ogs-dn}.
\end{itemize}
Then, $$\pi=\pi_1\cdot w_{q}=\begin{cases}
w_{q}\cdot \pi_1 & \text{if}\quad maj(\pi_1)=k_1 \\
w_{maj(\pi_1)}\cdot w_{q}\cdot \pi_1 & \text{if}\quad maj(\pi_1)<k_1.
\end{cases}$$
\end{lemma}

\begin{proof}

The proof is by induction on $m$. We start with m=1:
By Proposition \ref{exchange-dn}:
$$w_{q} \cdot t_{k_1}^{i_{k_1}} = w_{i_{k_1}} \cdot t_{k_1}^{i_{k_1}} \cdot w_{q},$$
where, $k_1\leq q$. \\

By Lemma \ref{order-w},  $w_{i_{k_1}}^2=1$. Therefore, 
$$t_{k_1}^{i_{k_1}} \cdot w_{q}=w_{i_{k_1}}\cdot w_{q} \cdot t_{k_1}^{i_{k_1}}.$$

Now, assume by induction the lemma holds for $m=j$, \\ i.e.,
for $\pi_{1}=t_{k_1}^{i_{k_1}}\cdot t_{k_2}^{i_{k_2}}\cdots t_{k_j}^{i_{k_j}}$ a standard $OGS$ elementary element by considering $\pi_1$ as an element of $S_n$ and $q\geq k_j$, 
$$\pi_1\cdot w_{q}=\begin{cases}
w_{q}\cdot \pi_1 & \text{if}\quad maj(\pi_1)=k_1 \\
w_{maj(\pi_1)}\cdot w_{q}\cdot \pi_1 & \text{if}\quad maj(\pi_1)<k_1.
\end{cases}.$$

Now, consider $\pi_1=t_{k_1}^{i_{k_1}}\cdot t_{k_2}^{i_{k_2}}\cdots t_{k_j}^{i_{k_j}}\cdot t_{k_{j+1}}^{i_{k_{j+1}}}$ a standard $OGS$ elementary element by considering $\pi_1$ as an element of $S_n$.
Then, by \cite{S1}, Theorem 28, the subword $\pi_2=t_{k_2}^{i_{k_2}}\cdots t_{k_j}^{i_{k_j}}\cdot t_{k_{j+1}}^{i_{k_{j+1}}}$ of $\pi_1$ is a standard $OGS$ elementary element as well by considering $\pi_2$ as an element of $S_n$. Since, $\pi_2$ is a proper subword of $\pi_1$ (i.e., $\pi_2\neq \pi_1$), by \cite{S1}, Theorem 28, $maj(\pi_2)<k_2$.
Then, by the induction assumption:
$$\pi_1\cdot w_q=t_{k_1}^{i_{k_1}}\cdot t_{k_2}^{i_{k_2}}\cdots t_{k_{j+1}}^{i_{k_{j+1}}}\cdot w_q=t_{k_1}^{i_{k_1}}\cdot \pi_2\cdot w_q= t_{k_1}^{i_{k_1}}\cdot w_{maj(\pi_2)}\cdot w_q\cdot \pi_2$$ $$=(t_{k_1}^{i_{k_1}}\cdot w_q)\cdot w_{maj(\pi_2)}\cdot \pi_2=w_{i_{k_1}}\cdot w_q\cdot t_{k_1}^{i_{k_1}}\cdot w_{maj(\pi_2)}\cdot \pi_2 .$$

Notice,  $$maj(\pi_2)=\sum_{x=2}^{j+1}i_{k_x}< \sum_{x=1}^{j+1}i_{k_x}=maj(\pi_1).$$
Since, $\pi_1$ is a standard $OGS$ elementary element, by Definition \ref{elementary}, $maj(\pi_1)\leq k_1=i_{k_1}+maj(\pi_2)\leq k_1$.
Hence, $maj(\pi_2)<maj(\pi_1)\leq k_1$.
Then, by applying Proposition \ref{exchange-dn}, 
$$t_{k_1}^{i_{k_1}}\cdot w_{maj(\pi_2)}=\begin{cases}
w_{i_{k_1}}\cdot t_{k_1}^{i_{k_1}} & \text{if}\quad i_{k_1}+maj(\pi_2)=maj(\pi_1)=k_1 \\
w_{i_{k_1}}\cdot w_{maj(\pi_1)}\cdot t_{k_1}^{i_{k_1}} & \text{if}\quad i_{k_1}+maj(\pi_2)=maj(\pi_1)<k_1.
\end{cases}.$$
Hence we get
$$\pi_1\cdot w_q = w_{i_{k_1}}\cdot w_q\cdot (t_{k_1}^{i_{k_1}}\cdot w_{maj(\pi_2)})\cdot \pi_2= $$ $$=\begin{cases}
w_{i_{k_1}}\cdot w_q\cdot w_{i_{k_1}}\cdot t_{k_1}^{i_{k_1}}\cdot \pi_2 & \text{if}\quad maj(\pi_1)=k_1 \\
w_{i_{k_1}}\cdot w_q\cdot w_{i_{k_1}}\cdot w_{maj(\pi_1)}\cdot t_{k_1}^{i_{k_1}}\cdot \pi_2 & \text{if}\quad maj(\pi_1)<k_1.
\end{cases}.$$
Since by Lemma \ref{order-w}, $w_{i_{k_1}}^2=1$ and by the definition of $\pi_2$ in the lemma we have $t_{k_1}^{i_{k_1}}\cdot \pi_2=\pi_1$.
Hence,
$$w_{i_{k_1}}\cdot w_q\cdot w_{i_{k_1}}\cdot (t_{k_1}^{i_{k_1}}\cdot \pi_2)=w_q\cdot \pi_1\quad w_{i_{k_1}}\cdot w_q\cdot w_{i_{k_1}}\cdot w_{maj(\pi_1)}\cdot (t_{k_1}^{i_{k_1}}\cdot \pi_2)=w_{maj(\pi_1)}\cdot w_q\cdot \pi_1.$$ 
Hence, we get for $\pi_1=t_{k_1}^{i_{k_1}}\cdot t_{k_2}^{i_{k_2}}\cdots t_{k_j}^{i_{k_j}}\cdot t_{k_{j+1}}^{i_{k_{j+1}}}:$
$$\pi_1\cdot w_{q}=\begin{cases}
w_{q}\cdot \pi_1 & \text{if}\quad maj(\pi_1)=k_1 \\
w_{maj(\pi_1)}\cdot w_{q}\cdot \pi_1 & \text{if}\quad maj(\pi_1)<k_1.
\end{cases}.$$
Hence, the theorem holds for every $\pi=\pi_1\cdot w_q$ such that $\pi_1$ is a standard $OGS$ elementary element by considering $\pi_1$ as an element of $S_n$.

\end{proof}

\begin{definition}\label{maj-alpha}
Let $\pi\in D_n$. Consider the presentation of $\pi$ as it is presented in Corollary \ref{twtw}, with all the notations of the corollary. i.e.,  
  $$\pi= \pi_{1}^{\circ}\cdot \pi_{1}^{\bullet}\cdot \pi_{2}^{\circ}\cdot \pi_{2}^{\bullet}\cdots \pi_{\mu-1}^{\circ}\cdot \pi_{\mu-1}^{\bullet}\cdot \pi_{\mu}^{\circ}$$ $$= \pi_{1}^{\circ}\cdot w_{L_{1_1}} \cdots w_{L_{1_{\nu_{1}}}} \cdot t_{k_{r_{1}+1}}^{i_{k_{r_{1}+1}}} \cdots t_{k_{r_{2}}}^{i_{k_{r_{2}}}}\cdot w_{L_{2_1}} \cdots w_{L_{2_{\nu_{2}}}} \cdot t_{k_{r_{2}+1}}^{i_{k_{r_{2}}+1}} \cdots $$ $$\cdots  t_{k_{r_{\mu-1}}}^{i_{k_{r_{\mu-1}}}} \cdot  w_{L_{{\mu-1}_1}} \cdots w_{L_{{\mu-1}_{\nu_{\mu-1}}}} \cdot \pi_{\mu}^{\circ},$$

where,  
$$\text{either}\quad \pi_{1}^{\circ}=t_{k_{1}}^{i_{k_{1}}} \cdots t_{k_{r_{1}}}^{i_{k_{r_{1}}}}\quad \text{or}\quad \pi_{1}^{\circ}=1,$$

and, 
$$\text{either}\quad  \pi_{\mu}^{\circ}=t_{k_{r_{\mu-1}+1}}^{i_{k_{r_{\mu-1}+1}}} \cdots t_{k_{r_{\mu}}}^{i_{k_{r_{\mu}}}}\quad  \text{or}\quad \pi_{\mu}^{\circ}=1.$$
  
  Then for every $1\leq \alpha\leq \mu$, we define $maj_{\alpha}(\pi)$  and $\rho_{\alpha}(\pi)$ to be $$maj_{\alpha}(\pi)=maj(\prod_{j=1}^{\alpha}\pi_{j}^{\circ})=\sum_{j=1}^{r_{\alpha}} i_{k_{j}},$$
  $$\rho_{\alpha}(\pi)=maj(\pi^{\prime})-maj_{\alpha}(\pi)= maj(\prod_{j=\alpha+1}^{\mu}\pi_{j}^{\circ})=\sum_{j=r_{\alpha}+1}^{r_{\mu}} i_{k_{j}}.$$
 
  For every $L_{\alpha_j}$, where  $1\leq \alpha\leq \mu$,$1\leq j\leq \nu_{\alpha}$ we define $\varrho_{L_{\alpha_j}}(\pi)$ to be
  $$\varrho_{L_{\alpha_j}}(\pi)=\begin{cases}
  L_{\alpha_j}-\rho_{\alpha}(\pi) & \text{if}\quad L_{\alpha_j}\geq maj(\pi^{\prime}) \\
  L_{\alpha_j} & \text{if}\quad L_{\alpha_j} < maj(\pi^{\prime})
  \end{cases}$$
  
\end{definition}

\begin{theorem}\label{wt}
 Let $\pi\in D_n$ such that $\pi^{\prime}$ is a standard $OGS$ elementary element of $S_n$. Consider the presentation of $\pi$ as it presented in Corollary \ref{twtw}, with all the notations. i.e.,  
  
  $$\pi= \pi_{1}^{\circ}\cdot \pi_{1}^{\bullet}\cdot \pi_{2}^{\circ}\cdot \pi_{2}^{\bullet}\cdots \pi_{\mu-1}^{\circ}\cdot \pi_{\mu-1}^{\bullet}\cdot \pi_{\mu}^{\circ}$$ $$= \pi_{1}^{\circ}\cdot w_{L_{1_1}} \cdots w_{L_{1_{\nu_{1}}}} \cdot t_{k_{r_{1}+1}}^{i_{k_{r_{1}+1}}} \cdots t_{k_{r_{2}}}^{i_{k_{r_{2}}}}\cdot w_{L_{2_1}} \cdots w_{L_{2_{\nu_{2}}}} \cdot t_{k_{r_{2}+1}}^{i_{k_{r_{2}}+1}} \cdots $$ $$\cdots  t_{k_{r_{\mu-1}}}^{i_{k_{r_{\mu-1}}}} \cdot  w_{L_{{\mu-1}_1}} \cdots w_{L_{{\mu-1}_{\nu_{\mu-1}}}} \cdot \pi_{\mu}^{\circ},$$

where,  
$$\text{either}\quad \pi_{1}^{\circ}=t_{k_{1}}^{i_{k_{1}}} \cdots t_{k_{r_{1}}}^{i_{k_{r_{1}}}}\quad \text{or}\quad \pi_{1}^{\circ}=1,$$

and, 
$$\text{either}\quad  \pi_{\mu}^{\circ}=t_{k_{r_{\mu-1}+1}}^{i_{k_{r_{\mu-1}+1}}} \cdots t_{k_{r_{\mu}}}^{i_{k_{r_{\mu}}}}\quad  \text{or}\quad \pi_{\mu}^{\circ}=1.$$

  Then, $\pi$ can be presented uniquely in the form: 
  $$\pi=\pi^{\bullet}\cdot \pi^{\circ}\quad \text{such that}\quad \pi^{\bullet}\in Id^{\bullet}_n, \quad \pi^{\circ}\in S^{\circ}_n$$
  
  ($Id^{\bullet}_n$ and $S^{\circ}_n$ are defined in Theorem \ref{subgroup-t-w})\\
  
  where,
  
  $$\pi^{\bullet} = \prod_{\alpha=1}^{\mu-1} w_{maj_{\alpha ~|~ maj_{\alpha}(\pi)<k_1 }(\pi)}^{-0.5 \cdot (-1)^{\nu_{\alpha}}+0.5}\cdot w_{L_{1_1}} \cdots w_{L_{1_{\nu_1}}} \cdot w_{L_{2_1}} \cdots w_{L_{2_{\nu_2}}}  \cdots w_{L_{{\mu-1}_1}} \cdots w_{L_{{\mu-1}_{\nu_{\mu-1}}}},$$
  
  where, for every $\pi\in D_n$, $maj_{\alpha}(\pi)$  for $1\leq \alpha \leq \mu-1$ as it is defined in Definition \ref{maj-alpha}.
  
  $$\pi^{\circ} = \prod_{j=1}^{r_{\mu}} t_{k_j}^{i_{k_{j}}}.$$
  
Notice, considering $\pi^{\circ}$ as an element of $S_n$, the standard $OGS$ presentation (as it is presented in Theorem \ref{canonical-sn}) of $\pi^{\circ}$ is the same to the standard $OGS$ presentation of $\pi^{\prime}$ (i.e., $\pi^{\circ}$ and $\pi^{\prime}$ present the same element of $S_n$.).

\end{theorem}

\begin{proof}

The proof is by induction on $\mu$.
If $\mu=1$, then $\pi=\pi_1^{\circ}$, and then the theorem holds trivially. Hence, we start with $\mu=2$.
If $\mu=2$ and $\pi_1^{\circ}=1$, then $\pi^{\bullet}=\pi_1^{\bullet}$ and $\pi^{\circ}=\pi_2^{\circ}$. Hence, the theorem holds trivially. Therefore, let assume $\mu=2$ and $\pi_1^{\circ}\neq 1$.
Then, $$\pi=\pi_1^{\circ}\cdot \pi_1^{\bullet}\cdot \pi_2^{\circ}=\begin{cases}
t_{k_{1}}^{i_{k_{1}}} \cdots t_{k_{r_{1}}}^{i_{k_{r_{1}}}}\cdot w_{L_{1_1}} \cdots w_{L_{1_{\nu_{1}}}} \cdot t_{k_{r_{1}+1}}^{i_{k_{r_{1}+1}}} \cdots t_{k_{r_{2}}}^{i_{k_{r_{2}}}} & \text{if}\quad \pi_2^{\circ}\neq 1 \\ 
t_{k_{1}}^{i_{k_{1}}} \cdots t_{k_{r_{1}}}^{i_{k_{r_{1}}}}\cdot w_{L_{1_1}} \cdots w_{L_{1_{\nu_{1}}}} & \text{if}\quad \pi_2^{\circ}= 1,\end{cases}$$
where,
\begin{itemize}
    \item $k_{r_1}\leq L_{1_1}$;
    \item $L_{1_u}<L_{1_{u+1}}$ for $1\leq u\leq \nu_1-1$;
    \item $L_{1_{\nu_{1}}}<k_{r_{1}+1}$.
\end{itemize}
Since by Corollary \ref{w-reduce},  the standard $OGS$ presentation (as it is  described in Theorem \ref{canonical-sn}) of $\pi^{\prime}$ is the same to the generalized standard $OGS$ presentation of   $\pi_1^{\circ}\cdot \pi_2^{\circ}$ (as it is described in Theorem \ref{ogs-dn}), we get that  by considering $\pi_1^{\circ}\cdot \pi_2^{\circ}$ as an element of $S_n$, it  presents the same element to $\pi^{\prime}$. Therefore,  $\pi_1^{\circ}\cdot \pi_2^{\circ}$ is a standard $OGS$ elementary element by considering it as an element of $S_n$, and then   by \cite{S1}, Theorem 28, $\pi_1^{\circ}$ and $\pi_2^{\circ}$ are standard $OGS$ elementary elements as well, by considering  $\pi_1^{\circ}$ and $\pi_2^{\circ}$ as elements of $S_n$.
Then, by applying Lemma \ref{tw-standard} on $\pi^{\circ}\cdot w_{L_{1_u}}$ for $1\leq u\leq \nu_1$, we have:
$$\pi_1^{\circ}\cdot w_{L_{1_u}}=\begin{cases}
w_{maj(\pi_1^{\circ})}\cdot  w_{L_{1_u}}\cdot  \pi_1^{\circ} & \text{if}\quad maj(\pi_1^{\circ})<k_1 \\
 w_{L_{1_u}}\cdot  \pi_1^{\circ} & \text{if}\quad maj(\pi_1^{\circ})=k_1
\end{cases}$$

Since by Definition \ref{maj-alpha},  $maj_1(\pi)=maj(\pi_1^{\circ})$, we get
$$\pi=\pi_1^{\circ}\cdot \pi_1^{\bullet}\cdot \pi_2^{\circ}=\begin{cases}
t_{k_{1}}^{i_{k_{1}}} \cdots t_{k_{r_{1}}}^{i_{k_{r_{1}}}}\cdot w_{L_{1_1}} \cdots w_{L_{1_{\nu_{1}}}} \cdot t_{k_{r_{1}+1}}^{i_{k_{r_{1}+1}}} \cdots t_{k_{r_{2}}}^{i_{k_{r_{2}}}} & \text{if}\quad \pi_2^{\circ}\neq 1 \\ 
t_{k_{1}}^{i_{k_{1}}} \cdots t_{k_{r_{1}}}^{i_{k_{r_{1}}}}\cdot w_{L_{1_1}} \cdots w_{L_{1_{\nu_{1}}}} & \text{if}\quad  \pi_2^{\circ}= 1,\end{cases}$$
$$=\begin{cases}
(w_{maj_1}(\pi))^{\nu_1}\cdot w_{L_{1_1}} \cdots w_{L_{1_{\nu_{1}}}}\cdot \pi_1^{\circ}\cdot \pi_2^{\circ} & \text{if}\quad \pi_2^{\circ}\neq 1 \\
(w_{maj_1}(\pi))^{\nu_1}\cdot w_{L_{1_1}} \cdots w_{L_{1_{\nu_{1}}}}\cdot \pi_1^{\circ} & \text{if}\quad \pi_2^{\circ}= 1\quad \text{and}\quad maj_1(\pi)<k_1 \\
w_{L_{1_1}} \cdots w_{L_{1_{\nu_{1}}}}\cdot \pi_1^{\circ} & \text{if}\quad \pi_2^{\circ}= 1\quad \text{and}\quad maj_1(\pi)=k_1.
\end{cases}$$
\\
Hence, the theorem holds for $\mu=2$.\\

Now, assume by induction the theorem holds for $\mu=\alpha$, and we prove it for $\mu=\alpha+1$.\\

Hence, we assume by induction:
$$\pi= \pi_{1}^{\circ}\cdot \pi_{1}^{\bullet}\cdot \pi_{2}^{\circ}\cdot \pi_{2}^{\bullet}\cdots \pi_{\alpha-1}^{\circ}\cdot \pi_{\alpha-1}^{\bullet}\cdot \pi_{\alpha}^{\circ}$$ $$= \pi_{1}^{\circ}\cdot w_{L_{1_1}} \cdots w_{L_{1_{\nu_{1}}}} \cdot t_{k_{r_{1}+1}}^{i_{k_{r_{1}+1}}} \cdots t_{k_{r_{2}}}^{i_{k_{r_{2}}}}\cdot w_{L_{2_1}} \cdots w_{L_{2_{\nu_{2}}}} \cdot t_{k_{r_{2}+1}}^{i_{k_{r_{2}}+1}} \cdots $$ $$\cdots  t_{k_{r_{\alpha-1}}}^{i_{k_{r_{\alpha-1}}}} \cdot  w_{L_{{\alpha-1}_1}} \cdots w_{L_{{\alpha-1}_{\nu_{\alpha-1}}}} \cdot \pi_{\alpha}^{\circ}$$
$$=\prod_{u=1}^{\alpha-1} \big(w_{maj_{u ~|~ maj_{u}(\pi)<k_1 }(\pi)}\big)^{\nu_{u}}\cdot w_{L_{1_1}} \cdots w_{L_{1_{\nu_1}}} \cdot w_{L_{2_1}} \cdots w_{L_{2_{\nu_2}}}  \cdots$$ $$\cdots w_{L_{{\alpha-1}_1}} \cdots w_{L_{{\alpha-1}_{\nu_{\alpha-1}}}}\cdot \pi_{1}^{\circ}\cdot  \pi_{2}^{\circ}\cdots \pi_{\alpha}^{\circ}.$$ 
where,
$$\text{either}\quad \pi_{1}^{\circ}=t_{k_{1}}^{i_{k_{1}}} \cdots t_{k_{r_{1}}}^{i_{k_{r_{1}}}}\quad \text{or}\quad \pi_{1}^{\circ}=1,$$

and
$$\text{either}\quad  \pi_{\alpha}^{\circ}=t_{k_{r_{\alpha-1}+1}}^{i_{k_{r_{\alpha-1}+1}}} \cdots t_{k_{r_{\alpha}}}^{i_{k_{r_{\alpha}}}}\quad  \text{or}\quad \pi_{\alpha}^{\circ}=1.$$
and 
\begin{itemize}
    \item $k_{r_u}\leq L_{u_1}$ for $1\leq u\leq \alpha-1$;
    \item $L_{u_v}<L_{u_{v+1}}$ for $1\leq u\leq \alpha-1$, and $1\leq v\leq \nu_u-1$;
    \item $L_{u_{\nu_{u}}}<k_{r_{u}+1}$ $1\leq u\leq \alpha-1$ in case $\pi_{\alpha}^{\circ}\neq 1$;
    \item $L_{u_{\nu_{u}}}<k_{r_{u}+1}$ $1\leq u\leq \alpha-2$ in case $\pi_{\alpha}^{\circ}=1$ (Since, there is no $k_{r_{\alpha-1}+1}$ in case $\pi_{\alpha}^{\circ}=1$).
\end{itemize}

Now, consider $\mu=\alpha+1$.
Then, 
$$\pi= \pi_{1}^{\circ}\cdot \pi_{1}^{\bullet}\cdot \pi_{2}^{\circ}\cdot \pi_{2}^{\bullet}\cdots \pi_{\alpha-1}^{\circ}\cdot \pi_{\alpha-1}^{\bullet}\cdot  \pi_{\alpha}^{\circ}\cdot \pi_{\alpha}^{\bullet}\cdot\pi_{\alpha+1}^{\circ}$$ $$= \pi_{1}^{\circ}\cdot w_{L_{1_1}} \cdots w_{L_{1_{\nu_{1}}}} \cdot t_{k_{r_{1}+1}}^{i_{k_{r_{1}+1}}} \cdots t_{k_{r_{2}}}^{i_{k_{r_{2}}}}\cdot w_{L_{2_1}} \cdots w_{L_{2_{\nu_{2}}}} \cdot t_{k_{r_{2}+1}}^{i_{k_{r_{2}}+1}} \cdots $$ $$\cdots  t_{k_{r_{\alpha}}}^{i_{k_{r_{\alpha}}}} \cdot  w_{L_{{\alpha}_1}} \cdots w_{L_{{\alpha}_{\nu_{\alpha}}}} \cdot \pi_{\alpha+1}^{\circ},$$

where,
$$\text{either}\quad \pi_{1}^{\circ}=t_{k_{1}}^{i_{k_{1}}} \cdots t_{k_{r_{1}}}^{i_{k_{r_{1}}}}\quad \text{or}\quad \pi_{1}^{\circ}=1,$$

and
$$\text{either}\quad  \pi_{\alpha+1}^{\circ}=t_{k_{r_{\alpha}+1}}^{i_{k_{r_{\alpha}+1}}} \cdots t_{k_{r_{\alpha+1}}}^{i_{k_{r_{\alpha+1}}}}\quad  \text{or}\quad \pi_{\alpha+1}^{\circ}=1.$$
and 
\begin{itemize}
    \item $k_{r_u}\leq L_{u_1}$ for $1\leq u\leq \alpha$;
    \item $L_{u_v}<L_{u_{v+1}}$ for $1\leq u\leq \alpha$, and $1\leq v\leq \nu_u-1$;
    \item $L_{u_{\nu_{u}}}<k_{r_{u}+1}$ $1\leq u\leq \alpha$ in case $\pi_{\alpha+1}^{\circ}\neq 1$;
    \item $L_{u_{\nu_{u}}}<k_{r_{u}+1}$ $1\leq u\leq \alpha-1$ in case $\pi_{\alpha+1}^{\circ}=1$ (Since, there is no $k_{r_{\alpha}+1}$ in case $\pi_{\alpha+1}^{\circ}=1$).
\end{itemize}

In case $\pi_{1}^{\circ}=1$, 
$$\pi=\pi_{1}^{\bullet}\cdot \pi_{2}^{\circ}\cdot \pi_{2}^{\bullet}\cdots \pi_{\alpha-1}^{\circ}\cdot \pi_{\alpha-1}^{\bullet}\cdot  \pi_{\alpha}^{\circ}\cdot \pi_{\alpha}^{\bullet}\cdot\pi_{\alpha+1}^{\circ}$$ $$=w_{L_{1_1}} \cdots w_{L_{1_{\nu_{1}}}} \cdot \pi_{2}^{\circ}\cdot \pi_{2}^{\bullet}\cdots \pi_{\alpha-1}^{\circ}\cdot \pi_{\alpha-1}^{\bullet}\cdot  \pi_{\alpha}^{\circ}\cdot \pi_{\alpha}^{\bullet}\cdot\pi_{\alpha+1}^{\circ}.$$
Notice,
$$\pi_{2}^{\circ}\cdot \pi_{2}^{\bullet}\cdots \pi_{\alpha-1}^{\circ}\cdot \pi_{\alpha-1}^{\bullet}\cdot  \pi_{\alpha}^{\circ}\cdot \pi_{\alpha}^{\bullet}\cdot\pi_{\alpha+1}^{\circ}=$$ $$=\prod_{u=2}^{\alpha} \big(w_{maj_{u ~|~ maj_{u}(\pi)<k_1 }(\pi)}\big)^{\nu_{u}}\cdot w_{L_{2_1}} \cdots w_{L_{2_{\nu_2}}} \cdot w_{L_{3_1}} \cdots w_{L_{3_{\nu_3}}}  \cdots$$ $$\cdots w_{L_{{\alpha}_1}} \cdots w_{L_{{\alpha}_{\nu_{\alpha}}}}\cdot \pi_{2}^{\circ}\cdots \pi_{\alpha+1}^{\circ}.$$ 
by the induction hypothesis.

Hence, in case $\pi_{1}^{\circ}=1$ we have: 
$$\pi=\pi_{1}^{\circ}\cdot \pi_{1}^{\bullet}\cdot \pi_{2}^{\circ}\cdot \pi_{2}^{\bullet}\cdots \pi_{\alpha-1}^{\circ}\cdot \pi_{\alpha-1}^{\bullet}\cdot  \pi_{\alpha}^{\circ}\cdot \pi_{\alpha}^{\bullet}\cdot\pi_{\alpha+1}^{\circ}=$$ $$=\prod_{u=1}^{\alpha} \big(w_{maj_{u ~|~ maj_{u}(\pi)<k_1 }(\pi)}\big)^{\nu_{u}}\cdot w_{L_{1_1}} \cdots w_{L_{1_{\nu_1}}} \cdot w_{L_{2_1}} \cdots w_{L_{2_{\nu_2}}}  \cdots$$ $$\cdots w_{L_{{\alpha}_1}} \cdots w_{L_{{\alpha}_{\nu_{\alpha}}}}\cdot \pi_{1}^{\circ}\cdot \pi_{2}^{\circ}\cdots \pi_{\alpha+1}^{\circ}.$$  

Hence, the theorem holds for $\mu=\alpha+1$ in case $\pi_{1}^{\circ}=1$.\\
Now, we consider 
$$\pi = \pi_{1}^{\circ}\cdot \pi_{1}^{\bullet}\cdot \pi_{2}^{\circ}\cdot \pi_{2}^{\bullet}\cdots \pi_{\alpha-1}^{\circ}\cdot \pi_{\alpha-1}^{\bullet}\cdot  \pi_{\alpha}^{\circ}\cdot \pi_{\alpha}^{\bullet}\cdot\pi_{\alpha+1}^{\circ}$$
where $$\pi_1^{\circ}\neq 1.$$
Then, 
$$\pi=t_{k_{1}}^{i_{k_{1}}} \cdots t_{k_{r_{1}}}^{i_{k_{r_{1}}}}\cdot w_{L_{1_1}} \cdots w_{L_{1_{\nu_{1}}}} \cdot t_{k_{r_{1}+1}}^{i_{k_{r_{1}+1}}} \cdots t_{k_{r_{2}}}^{i_{k_{r_{2}}}}\cdot w_{L_{2_1}} \cdots w_{L_{2_{\nu_{2}}}} \cdot t_{k_{r_{2}+1}}^{i_{k_{r_{2}}+1}} \cdots $$ $$\cdots  t_{k_{r_{\alpha}}}^{i_{k_{r_{\alpha}}}} \cdot  w_{L_{{\alpha}_1}} \cdots w_{L_{{\alpha}_{\nu_{\alpha}}}} \cdot \pi_{\alpha+1}^{\circ}$$

By the induction hypothesis:
$$\pi= \prod_{u=1}^{\alpha-1} \big(w_{maj_{u ~|~ maj_{u}(\pi)<k_1 }(\pi)}\big)^{\nu_{u}}\cdot w_{L_{1_1}} \cdots w_{L_{1_{\nu_1}}} \cdot w_{L_{2_1}} \cdots w_{L_{2_{\nu_2}}}  \cdots$$ $$\cdots w_{L_{{\alpha-1}_1}} \cdots w_{L_{{\alpha-1}_{\nu_{\alpha-1}}}}\cdot \pi_{1}^{\circ}\cdot  \pi_{2}^{\circ}\cdots \pi_{\alpha}^{\circ}\cdot w_{L_{{\alpha}_1}} \cdots w_{L_{{\alpha}_{\nu_{\alpha}}}} \cdot \pi_{\alpha+1}^{\circ}.$$

Now, consider   $$(\prod_{j=1}^{\alpha}\pi_j^{\circ})\cdot w_{L_{\alpha_k}}$$ for every  $1\leq k\leq \nu_{\alpha}$. Then by Lemma \ref{tw-standard} we conclude:

$$(\prod_{j=1}^{\alpha}\pi_j^{\circ})\cdot w_{L_{\alpha_k}}=\begin{cases}
w_{maj_{\alpha}(\pi)}\cdot  w_{L_{\alpha_k}}\cdot \prod_{j=1}^{\alpha} \pi_j^{\circ} & \text{if}\quad maj_{\alpha}(\pi)<k_1 \\
 w_{L_{\alpha_k}}\cdot \prod_{j=1}^{\alpha} \pi_j^{\circ} & \text{if}\quad maj_{\alpha}(\pi)=k_1
\end{cases}$$

Hence, we get:   

$$\pi = \prod_{u=1}^{\alpha}w_{maj_{u ~|~ maj_{u}(\pi)<k_1}(\pi)}^{\nu_{u}}\cdot w_{L_{1_1}} \cdots w_{L_{1_{\nu_1}}} \cdot w_{L_{2_1}} \cdots w_{L_{2_{\nu_2}}}  \cdots w_{L_{{\mu-1}_1}} \cdots w_{L_{{\mu-1}_{\nu_{\mu-1}}}}  \cdot$$ $$\cdot \pi_{1}^{\circ}\cdot  \pi_{2}^{\circ}\cdots \pi_{\alpha}^{\circ}\cdot \pi_{\alpha+1}^{\circ}.$$

Now, since by Lemma \ref{order-w},  $w_{maj_{\alpha}(\pi)}^2=1$ for all $1\leq \alpha\leq \mu-1$, 
and since the following holds:
$$
-0.5 \cdot (-1)^{\nu_{\alpha}}+0.5=\left\{\begin{array}{ll}
0 & \quad \text{if}\quad  \nu_{\alpha} \mod 2 =0
\\ \\
1 & \quad \text{if}\quad  \nu_{\alpha} \mod 2 =1

\end{array}\right.
$$
we get the result of the theorem.

\end{proof}

Now, we give the combinatorial meaning of Theorem \ref{wt} by considering $\pi\in D_n$ as a sign-permutation.

\begin{corollary}\label{wt-perm}
Let $\pi\in D_n$ such that $\pi^{\prime}$ is a standard $OGS$ elementary element of $S_n$. Consider the presentation of $\pi$ as it presented in Corollary \ref{twtw}, with all the notations. i.e.,  
  
  $$\pi= \pi_{1}^{\circ}\cdot \pi_{1}^{\bullet}\cdot \pi_{2}^{\circ}\cdot \pi_{2}^{\bullet}\cdots \pi_{\mu-1}^{\circ}\cdot \pi_{\mu-1}^{\bullet}\cdot \pi_{\mu}^{\circ}$$ $$= \pi_{1}^{\circ}\cdot w_{L_{1_1}} \cdots w_{L_{1_{\nu_{1}}}} \cdot t_{k_{r_{1}+1}}^{i_{k_{r_{1}+1}}} \cdots t_{k_{r_{2}}}^{i_{k_{r_{2}}}}\cdot w_{L_{2_1}} \cdots w_{L_{2_{\nu_{2}}}} \cdot t_{k_{r_{2}+1}}^{i_{k_{r_{2}}+1}} \cdots $$ $$\cdots  t_{k_{r_{\mu-1}}}^{i_{k_{r_{\mu-1}}}} \cdot  w_{L_{{\mu-1}_1}} \cdots w_{L_{{\mu-1}_{\nu_{\mu-1}}}} \cdot \pi_{\mu}^{\circ},$$

where,  
$$\text{either}\quad \pi_{1}^{\circ}=t_{k_{1}}^{i_{k_{1}}} \cdots t_{k_{r_{1}}}^{i_{k_{r_{1}}}}\quad \text{or}\quad \pi_{1}^{\circ}=1,$$

and, 
$$\text{either}\quad  \pi_{\mu}^{\circ}=t_{k_{r_{\mu-1}+1}}^{i_{k_{r_{\mu-1}+1}}} \cdots t_{k_{r_{\mu}}}^{i_{k_{r_{\mu}}}}\quad  \text{or}\quad \pi_{\mu}^{\circ}=1.$$
For $1\leq \alpha\leq \mu-1$, let $maj_{\alpha}(\pi)$ be as it is defined in Definition $maj_{\alpha}(\pi)$.
The subset $N(\pi)$ of the $\{1, 2,\ldots, n\}$ is defined as follow:
\begin{itemize}
    \item For $1\leq \alpha\leq \mu-1$,  $maj_{\alpha}(\pi)+1\in N(\pi)$ if and only if  $\nu_{\alpha}$ (as it is defined in Corollary \ref{twtw}) is odd and $maj_{\alpha}(\pi)<k_1$.
    \item For $1\leq \alpha\leq \mu-1$ and $1\leq j\leq \nu_{\alpha}$, ~~$L_{\alpha_j}+1\in N(\pi)$.
    \item  If $2\leq x\leq n$ and $x\neq maj_{\alpha}(\pi)+1$ and $x\neq L_{\alpha_j}+1$ for any $1\leq \alpha\leq \mu-1$ and any $1\leq j\leq \nu_{\alpha}$, then $x\notin N(\pi)$.
    \item $1\in N(\pi)$ if and only if the number of elements $x$ such that $2\leq x\leq n$ and $x\in N(\pi)$ is odd.
    \end{itemize}
Then, $\pi$ is the following sign-permutation of $D_n$:
$$\begin{cases} 
\pi(x)=\pi^{\prime}(x) & \text{if}\quad x\notin N(\pi) \\
\pi(x)=-\pi^{\prime}(x) & \text{if}\quad x\in N(\pi).
\end{cases}$$
\end{corollary}

\begin{proof}
Let $$\pi=\pi_{1}^{\circ}\cdot w_{L_{1_1}} \cdots w_{L_{1_{\nu_{1}}}} \cdot t_{k_{r_{1}+1}}^{i_{k_{r_{1}+1}}} \cdots t_{k_{r_{2}}}^{i_{k_{r_{2}}}}\cdot w_{L_{2_1}} \cdots w_{L_{2_{\nu_{2}}}} \cdot t_{k_{r_{2}+1}}^{i_{k_{r_{2}}+1}} \cdots $$ $$\cdots  t_{k_{r_{\mu-1}}}^{i_{k_{r_{\mu-1}}}} \cdot  w_{L_{{\mu-1}_1}} \cdots w_{L_{{\mu-1}_{\nu_{\mu-1}}}} \cdot \pi_{\mu}^{\circ},$$

where,  
$$\text{either}\quad \pi_{1}^{\circ}=t_{k_{1}}^{i_{k_{1}}} \cdots t_{k_{r_{1}}}^{i_{k_{r_{1}}}}\quad \text{or}\quad \pi_{1}^{\circ}=1,$$

and, 
$$\text{either}\quad  \pi_{\mu}^{\circ}=t_{k_{r_{\mu-1}+1}}^{i_{k_{r_{\mu-1}+1}}} \cdots t_{k_{r_{\mu}}}^{i_{k_{r_{\mu}}}}\quad  \text{or}\quad \pi_{\mu}^{\circ}=1.$$
Assume, $\pi^{\prime}$ is a standard $OGS$ elementary element as it is defined in Definition \ref{elementary}. Then, by Theorem \ref{wt}, 
$$\pi=\pi^{\bullet}\cdot \pi^{\circ},$$
such that $$\pi^{\bullet}= \prod_{u=1}^{\alpha}w_{maj_{u ~|~ maj_{u}(\pi)<k_1}(\pi)}^{\nu_{u}}\cdot w_{L_{1_1}} \cdots w_{L_{1_{\nu_1}}} \cdot w_{L_{2_1}} \cdots w_{L_{2_{\nu_2}}}  \cdots w_{L_{{\mu-1}_1}} \cdots w_{L_{{\mu-1}_{\nu_{\mu-1}}}}$$
and $$\pi^{\circ}=\pi_1^{\circ}\cdot \pi_{2}^{\circ}\cdots \pi_{\mu}^{\circ}$$
where, by Theorem \ref{wt}, the presentation of $\pi^{\circ}$ in terms of generalized standard $OGS$ (as it is presented in Theorem \ref{ogs-dn}) is same to the standard $OGS$ presentation of $\pi^{\prime}$ (as it is presented in Theorem \ref{canonical-sn}). Hence, for $1\leq x\leq n$ (i.e., $x>0$) $\pi^{\circ}(x)=\pi^{\prime}(x)$.
Since, by Theorem \ref{wt}, $\pi^{\bullet}\in Id_n^{\bullet}$, and by Theorem \ref{subgroup-t-w}, $\big(Id_n^{\bullet}\big)^{\prime}=1$, we get by Definition \ref{hom-dn-sn}, $|\pi^{\bullet}(x)|=x$ for $1\leq x\leq n$ (i.e., either $\pi^{\bullet}(x)=x$ or 
$\pi^{\bullet}(x)=-x$ for $1\leq x\leq n$.
By Remark \ref{perm-w},  for every $1\leq x\leq n-1$, $w_x$ satisfies the following properties $$\begin{array}{l}
w_{x}(1)=-1 \\
w_{x}(x+1)=-(x+1) \\
w_{x}(j)=j, \text { for } 2\leqslant j \leqslant x \quad \text{and} \quad x+2\leqslant j \leqslant n.
\end{array}.$$
Hence,
for every $1\leq \alpha\leq \mu-1$, $w_{maj_{\alpha}(\pi)}$ satisfies the following properties $$\begin{array}{l}
w_{maj_{\alpha}(\pi)}(1)=-1 \\
w_{maj_{\alpha}(\pi)}(maj_{\alpha}(\pi)+1)=-(maj_{\alpha}(\pi)+1) \\
w_{maj_{\alpha}(\pi)}(j)=j, \text { for } 2\leqslant j \leqslant maj_{\alpha}(\pi)\quad \text{and} \quad maj_{\alpha}+2\leqslant j \leqslant n.
\end{array}.$$
and for every $1\leq \alpha\leq \mu-1$,  and $1\leq q\leq \nu_{\alpha}$, $w_{L_{{\alpha}_q}}(\pi)$ satisfies the following  $$\begin{array}{l}
w_{L_{{\alpha}_q}}(1)=-1 \\
w_{L_{{\alpha}_q}}(L_{{\alpha}_q}+1)=-(L_{{\alpha}_q}+1) \\
w_{L_{{\alpha}_q}}(j)=j, \text { for } 2\leqslant j \leqslant L_{{\alpha}_q}\quad \text{and} \quad L_{{\alpha}_q}+2\leqslant j \leqslant n.
\end{array}.$$
Hence, for $2\leq x\leq n$,
\begin{itemize}
    \item $$\pi^{\bullet}(x)=$$ $$\prod_{u=1}^{\alpha}w_{maj_{u ~|~ maj_{u}(\pi)<k_1}(\pi)}^{\nu_{u}}\cdot w_{L_{1_1}} \cdots w_{L_{1_{\nu_1}}} \cdot w_{L_{2_1}} \cdots w_{L_{2_{\nu_2}}}  \cdots w_{L_{{\mu-1}_1}} \cdots w_{L_{{\mu-1}_{\nu_{\mu-1}}}}(x)=-x$$
    if one of the following holds:
    \begin{itemize}
        \item For $1\leq \alpha\leq \mu-1$,  $x=maj_{\alpha}(\pi)+1$ and  $\nu_{\alpha}$ is odd and $maj_{\alpha}(\pi)<k_1$.
        \item For $1\leq \alpha\leq \mu-1$ and $1\leq j\leq \nu_{\alpha}$, ~~$x=L_{\alpha_j}+1$.
    \end{itemize}
\item $$\pi^{\bullet}(x)=$$ $$\prod_{u=1}^{\alpha}w_{maj_{u ~|~ maj_{u}(\pi)<k_1}(\pi)}^{\nu_{u}}\cdot w_{L_{1_1}} \cdots w_{L_{1_{\nu_1}}} \cdot w_{L_{2_1}} \cdots w_{L_{2_{\nu_2}}}  \cdots w_{L_{{\mu-1}_1}} \cdots w_{L_{{\mu-1}_{\nu_{\mu-1}}}}(x)=x$$ for all the rest of values of $2\leq x\leq n$.
\end{itemize}

Notice, $\pi^{\bullet}(1)=-1$ if and only if the number of elements $x$ such that \\ $2\leq x\leq n$ and $\pi^{\bullet}(x)=-x$ is odd.\\

Notice, by the definition of $N(\pi)$,   $$\pi^{\bullet}(x)=-x\quad \text{if and only if}\quad x\in N(\pi).$$\\

Now, since for every $1\leq x\leq n$, $$\pi^{\circ}(x)=\pi^{\prime}(x), \quad \pi^{\circ}(-x)=-\pi^{\prime}(x),$$
we get 
$$\pi(x)=\pi^{\bullet}\cdot \pi^{\circ}(x)=\pi^{\prime}(x)\quad \text{if and only if}\quad \pi^{\bullet}(x)=x$$
and
$$\pi(x)=\pi^{\bullet}\cdot \pi^{\circ}(x)=-\pi^{\prime}(x)\quad \text{if and only if}\quad \pi^{\bullet}(x)=-x$$

\end{proof}

\begin{remark}\label{wt-general}
In case of $\pi\in D_n$ where $\pi^{\prime}$ is not a standard $OGS$ elementary element, there is possible to decompose $\pi$ in the form $$\pi=\pi^{\bullet}\cdot \pi^{\circ}$$ such that 
$$\pi^{\bullet}\in Id_n^{\bullet}\quad \quad \pi^{\circ}\in S_n^{\circ},$$
where $Id_n^{\bullet}$ and $S_n^{\circ}$ are subgroups of $D_n$ as are defined in Theorem \ref{subgroup-t-w}.
Since the formula for the description $\pi^{\bullet}$ is very complicated in case of general $\pi\in D_n$, we just describe briefly the algorithm for calculating $\pi^{\bullet}$. The standard $OGS$ presentation of $\pi^{\circ}$ is the same to the standard $OGS$ presentation of $\pi^{\prime}$ (i.e., we get the standard $OGS$ presentation of $\pi^{(\circ)}$ by just omitting the elements of the form $w_L$ from the generalized standard $OGS$ presentation of $\pi$).
Let $\pi\in D_n$ such that $\pi^{\prime}$ is not a standard $OGS$ elementary element of $S_n$. First, consider the standard $OGS$ elementary factorization of $\pi^{\prime}$ as it is defined in Definition \ref{canonical-factorization-def}: 
$$\pi^{\prime}=\prod_{v=1}^{z(\pi^{\prime})}\big({\pi^{\prime}}^{(v)}\big), ~~~~ where ~~~~\big({\pi^{\prime}}^{(v)}\big)=\prod_{j=1}^{m^{(v)}}t_{h^{(v)}_{j}}^{\imath_{j}^{(v)}},$$
 by the presentation in the standard $OGS$ canonical form  for every $1\leq v\leq z(\pi)$ and  $1\leq j\leq m^{(v)}$ such that:
 \begin{itemize}
\item $\imath_{j}^{(v)}>0;$ \\
\item $\sum_{j=1}^{m^{(1)}}\imath_{j}^{(1)}\leq h^{(1)}_{1}$ i.e., $maj\left({\pi^{\prime}}^{(1)}\right)\leq h^{(1)}_{1}$; \\
\item $h^{(v-1)}_{m^{(v-1)}}\leq\sum_{j=1}^{m^{(v)}}\imath_{j}^{(v)}\leq h^{(v)}_{1}$ for $2\leq v\leq z$ \\ \\
i.e., $h^{(v-1)}_{m^{(v-1)}}\leq maj\left[\left({\pi^{\prime}}^{(v)}\right)\right]\leq h^{(v)}_{1} ~~ for ~~ 2\leq v\leq z$.
\end{itemize}
Then, for every $1\leq v\leq z(\pi)$, let $\dot{\pi}^{(v)}$ be the subword of $\pi$ such the following holds:
\begin{itemize}
    \item For every $1\leq v\leq z(\pi)$, ~${[{\dot{\pi}}^{(v)}]}^{\prime}={\pi^{\prime}}^{(v)}$.
    \item  For $1\leq j\leq k_m$, $$w_j\in \dot{\pi}^{(v)}$$ if and only if 
$$\begin{cases}
maj\left[\left({\pi^{\prime}}^{(v)}\right)\right]\leq j<maj\left[\left({\pi^{\prime}}^{(v+1)}\right)\right] & \text{if}\quad 1\leq v\leq z(\pi)-1 \\ \\
maj\left[\left({\pi^{\prime}}^{(v)}\right)\right]\leq j & \text{if}\quad v=z(\pi).
\end{cases}$$
\end{itemize}

Then applying Theorem \ref{wt} on $\dot{\pi}^{(v)}$ for every $1\leq v\leq z(\pi)$.Then we get
$$\dot{\pi}^{(v)}=\left(\dot{\pi}^{(v)}\right)^{\bullet}\cdot\left(\dot{\pi}^{(v)}\right)^{\circ}.$$
Then by applying Proposition \ref{exchange-dn} and Lemma \ref{tw-standard} several times we move the elements of the form $w_j$ to the left hand side, and then we get the presentation $$\pi=\pi^{\bullet}\cdot \pi^{\circ}$$ such that $\pi^{\bullet}\in Id_n^{\bullet}$ and $\pi^{\circ}\in S_n^{\circ}$.
We  explain the process to get the decomposition  $\pi=\pi^{\bullet}\cdot \pi^{\circ}$  step by step in  Example \ref{general-length-example} .
\end{remark}

\section{The two parabolic subgroups of $D_{n}$ which are isomorphic to $S_n$}\label{two-para-sub} 

In Theorem \ref{subgroup-t-w}, we have already defined two subgroups of $D_n$, which are naturally characterized by the generalized standard $OGS$ presentation of $D_n$. The subgroup $S^{\circ}_n$ of $D_n$, which is isomorphic to $S_n$, and contains the elements of $D_n$ such that the generalized standard $OGS$ presentation of $\pi$ (as it is described in Theorem \ref{ogs-dn}) is the same to the standard $OGS$ presentation of $\pi^{\prime}$ (as it is described in Theorem \ref{canonical-sn}), and the subgroup $Id^{\bullet}_n$ of $D_n$, which is isomorphic to $\mathbb{Z}_2^{n-1}$, and contains all the elements of $D_n$ such that $\pi^{\prime}=1$. In this section we show the presentation of the elements in another interesting parabolic subgroup of $D_{n}$ which is isomorphic to $S_{n}$  by the generalized standard $OGS$ for $D_n$ . We start with the characterization of the two parabolic subgroups of $D_n$ which are isomorphic to $S_n$ (Notice, the first one has been defined in Theorem \ref{subgroup-t-w}).

\begin{itemize}
    \item The subgroup $S^{\circ}_{n}$ which is generated by $\{ s_{1}, s_{2}, \ldots, s_{n-1}\}.$
    \item The subgroup $S^{\circ}_{n^{\prime}}$ which is generated by $\{ s_{1^{\prime}}, s_{2}, \ldots, s_{n-1}\}.$
   \end{itemize}

Now, we find the generalized standard $OGS$ presentation for the elements of the subgroup $S^{\circ}_{n^{\prime}}$.
\begin{theorem}\label{s-nn-standard-ogs}
Let $\pi \in S^{\circ}_{n^{\prime}}$ such that $\pi^{\prime}$ is presented in the standard $OGS$ presentation $\pi^{\prime}=t_{k_{1}}^{i_{k_{1}}} \cdots t_{k_{m}}^{i_{k_{m}}}$ is a standard $OGS$ elementary element (as it is defined in Definition \ref{elementary}), Then  the decomposition of $\pi$ into the presentation $\pi=\pi^{\bullet}\cdot \pi^{\circ}$ by Theorem \ref{wt}, as follow:
\begin{enumerate}
    \item $\pi^{\bullet}=1$ ~and  ~$\pi=\pi^{\circ}=t_{k_{1}}^{i_{k_{1}}} \cdots t_{k_{m}}^{i_{k_{m}}}$  in case $\sum_{j=1}^{m} i_{k_{j}} = k_{1}$.
    \item $\pi^{\bullet} = w_{maj(\pi^{\prime})}$ ~and ~$\pi^{\circ}=t_{k_{1}}^{i_{k_{1}}} \cdots t_{k_{m}}^{i_{k_{m}}}$ in case $\sum_{j=1}^{m} i_{k_{j}} < k_{1}$.
\end{enumerate}
\end{theorem}
 \begin{proof}
 We divide the proof into the two mentioned cases. First, notice, by Theorem \ref{wt}, By considering $\pi^{\circ}$ as an element of $S_n$, $\pi^{\circ}$ presents the same element as $\pi^{\prime}$, and  the standard $OGS$ presentation (as it presented in Theorem \ref{canonical-sn}) and the normal form of $\pi^{\prime}$, is the same to the generalized standard $OGS$ presentation (as it presented in Theorem \ref{ogs-dn}) and the normal form of $\pi^{\circ}$. Hence, in the proof we use $maj(\pi^{\prime})$ and $norm(\pi^{\prime})$ when we deal with $\pi^{\circ}$.
 By using the normal form of standard $OGS$ elementary element of $S_n$ as it is described in \cite{S1}, Theorem 28. 
 \begin{equation}\label{norm-standard}
 norm(\pi^{\prime})=\prod_{u=\rho_1}^{k_1-1}\prod_{r=0}^{\rho_1-1}s_{u-r}\cdot \prod_{u=k_1}^{k_2-1}\prod_{r=0}^{\rho_2-1}s_{u-r}\cdot \prod_{u=k_2}^{k_3-1}\prod_{r=0}^{\rho_3-1}s_{u-r}\cdots \prod_{u=k_{m-1}}^{k_m-1}\prod_{r=0}^{\rho_m-1}s_{u-r},
 \end{equation}
where, $\rho_{j}=\sum_{x=j}^{m}i_{k_{x}}$ for $1\leq j\leq m$.
 
$\bullet$ We start with the first case where $\pi^{\prime} = t_{k_{1}}^{i_{k_{1}}} \cdots t_{k_{m}}^{i_{k_{m}}}$ is a standard OGS elementary element  and $\sum_{j=1}^{m} i_{k_{j}} = k_{1}$.\\
Since $\rho_1=\sum_{j=1}^m i_{k_j}=k_1$, Equation \ref{norm-standard} has the following form:
$$\text{norm} (\pi^{\prime})=\prod_{u=k_1}^{k_2-1}\prod_{r=0}^{\rho_2-1}s_{u-r}\cdot \prod_{u=k_2}^{k_3-1}\prod_{r=0}^{\rho_3-1}s_{u-r}\cdots \prod_{u=k_{m-1}}^{k_m-1}\prod_{r=0}^{\rho_m-1}s_{u-r},$$
Notice, the minimal value for $u-r$ in the formula of $norm(\pi^{\prime})$ occurs where $u=k_1$ and $r=\rho_2-1$. 
Then, 
$u-r=k_1-(\rho_2-1).$
Since $k_1=\sum_{x=1}^m i_{k_x}=i_1+\rho_2$, we get $\rho_2=k_1-i_{k_1}$.
Hence, $k_1-(\rho_2-1)=k_1-(k_1-i_{k_1}-1)=i_{k_1}+1$.
Now, since
$i_{k_{1}} \geq 1$, we have that 
$i_{k_{1}}+1 \geq 2$. 
Hence, the minimal value for $u-r$ in the formula for $norm(\pi^{\prime})$ is $u-r=2$, in case where $\sum_{j=1}^{m} i_{k_{j}} = k_{1}$.
Therefore, there is no occurrence of  $s_{1}$ in the normal form presentation of $\pi^{\prime}$ (which is by Theorem \ref{wt} is the same to the normal form of $\pi^{\circ}$).  Hence, $$\pi=\pi^{\circ}.$$\\

$\bullet$ Now, we turn to the second case:  $\pi^{\prime} = t_{k_{1}}^{i_{k_{1}}} \cdots t_{k_{m}}^{i_{k_{m}}}$ and $\sum_{j=1}^{m} i_{k_{j}} < k_{1}$.\\
Consider the normal form of $\pi^{\prime}$:
Since $\rho_1=\sum_{j=1}^m i_{k_j}<k_1$, Equation \ref{norm-standard} has the following form:
$$\text{norm}(\pi^{\prime}) = s_{maj(\pi^{\prime})} \cdots s_{1} \cdot 
s_{maj(\pi^{\prime})+1} \cdots s_{2} \cdots $$
Notice, that by the definition of $t_{j}$ for $2\leq j\leq n$, as it is defined in Theorem \ref{canonical-sn} as follow:
$$t_{maj(\pi^{\prime})+1}=s_1\cdot s_2\cdots s_{maj(\pi^{\prime})}.$$ Hence, 
$$t_{maj(\pi^{\prime})+1}\cdot \pi^{\circ}=t_{maj(\pi^{\prime}+1)}^{1}\cdot t_{k_1}^{i_{k_1}}\cdots t_{k_{m}}^{i_{k_m}}$$
$$=(s_1\cdot s_2\cdots s_{maj(\pi^{\prime})})\cdot (s_{maj(\pi^{\prime})} \cdots s_{1}) \cdot 
s_{maj(\pi^{\prime})+1} \cdots s_{2} \cdots $$
$$=s_{maj(\pi^{\prime})+1} \cdots s_{2} \cdots $$
Notice, that $\sum_{j=1}^m i_{k_j}+1=maj(\pi^{\prime})+1$. Hence, by the proof of the proposition for case 1, there is no occurrence of $s_1$ in the normal form of $t_{maj(\pi^{\prime})+1}\cdot \pi^{\prime}.$

Then by multiplying  $\pi^{\prime}$ by  $w_{{maj}(\pi^{\prime})}$ on the left hand side,  we have:
$$w_{maj(\pi^{\prime})} \cdot \pi^{\circ}= (s_{maj(\pi^{\prime})} \cdots s_2\cdot s_{1^{\prime}} \cdot s_{1} \cdot s_{2} \cdots s_{maj(\pi^{\prime})}) \cdot t_{maj(\pi^{\prime})+1}\cdot \pi^{\circ}$$
$$ = s_{maj(\pi^{\prime})} \cdots s_{1^{\prime}} \cdot t_{maj(\pi^{\prime})+1}\cdot \pi^{\circ}$$
Hence, we get
$$\pi= w_{maj(\pi^{\prime})} \cdot \pi^{\circ}.$$

 \end{proof}
\begin{example}

Consider $$\pi = t_{5}^{2} \cdot t_{6}^{3}$$
Then, by \cite{S1}, Theorem 28, the normal form of $\pi$ as follow: $$\text{norm} (\pi) = s_{5}\cdot s_{4} \cdot s_{3}.$$ 
Since there is no occurrence of $s_1$ in $\text{norm}(\pi)$, we get  $$\pi \in S^{\circ}_{n^{\prime}}.$$
\end{example}
\begin{example}
 Consider$$\pi = w_4\cdot t_{5}^{2} \cdot t_{6}^{2}$$
 Then, $$\pi^{\prime}= t_{5}^{2} \cdot t_{6}^{2}.$$
Then, by \cite{S1}, Theorem 28, the normal form of $\pi^{\prime}$ as follow: $$\text{norm}(\pi^{\prime})= (s_{4} \cdot s_{3} \cdot s_{2} \cdot s_{1}) \cdot (s_{5} \cdot s_{4}) $$
Notice, $$maj(\pi^{\prime})= 4$$ 
Hence, by Theorem \ref{s-nn-standard-ogs} $$\pi \in S^{\circ}_{n^{\prime}}.$$
Indeed, 
$$\pi=w_{4} \cdot t_{5}^{2} \cdot t_{6}^{2} = (s_{4} \cdot s_{3} \cdot s_{2} \cdot s_{1^{\prime}} \cdot s_{1} \cdot s_{2} \cdot s_{3} \cdot s_{4}) \cdot (s_{4} \cdot s_{3} \cdot s_{2} \cdot s_{1}) \cdot (s_{5} \cdot s_{4})$$ $$ = (s_{4} \cdot s_{3} \cdot s_{2} \cdot s_{1^{\prime}}) \cdot (s_{5} \cdot s_{4})=norm(\pi).$$

\end{example}
\begin{theorem}\label{parabolic-factorization}
 Let $\pi$ be an element of  $S^{\circ}_{n^{\prime}}$ such that  $\pi^{\prime}$ is not necessarily a standard $OGS$ elementary element. Consider the standard $OGS$ elementary factorization of $\pi^{\prime}$  as it is defined in Definition \ref{canonical-factorization-def} and then get $\pi^{\prime} = \pi_{1}^{\prime} \cdots \pi_{z}^{\prime}$, where the presentation of an element $\pi \in S^{\circ}_{n^{\prime}}$ by the generalized standard $OGS$ presentation as follow: 
 \begin{itemize}
     \item $\pi=\pi_{1}\cdot \pi_{2}\cdots \pi_{z}$;
     \item $\pi_{v} = \pi_{v}^{\prime}$ , where $\pi_{v}^{\prime} = t_{h_{v,1}}^{i_{h_{v,1}}} \cdot t_{h_{v,2}}^{i_{h_{v,2}}} \cdots t_{h_{v,m(\pi_{1})}}^{i_{h_{v,m(\pi_{1})}}} $ , where $\text{maj}(\pi_{v}^{\prime})= h_{v,1}$.
     \item $\pi_{v} = w_{maj(\pi_{v}^{\prime})} \cdot \pi_{v}^{\prime} $ ,where $\pi_{v}^{\prime} = t_{h_{v,1}}^{i_{h_{v,1}}} \cdot t_{h_{v,2}}^{i_{h_{v,2}}} \cdots t_{h_{v,m(\pi_{1})}}^{i_{h_{v,m(\pi_{1})}}} $ in case $\text{maj}(\pi_{v}^{\prime}) < h_{v,1}$.
 \end{itemize}

\end{theorem}

\begin{proof}
Look at the standard $OGS$ elementary factorization of $\pi^{\prime}$ (as it defined in Definition \ref{canonical-factorization-def}). Then we have, $$\pi^{\prime} = \pi_{1}^{\prime} \cdots \pi_{z}^{\prime}$$
where $\pi_{v}^{\prime}$ is a standard $OGS$ ekementary factor of $\pi^{\prime}$.
By  applying Theorem \ref{s-nn-standard-ogs} on every standard $OGS$ elementary factor $\pi_{v}^{\prime}$ (for $1\leq v\leq z$), we get the desired result.
\end{proof}
\begin{example}
Consider:
$$\pi = w_3\cdot t_{4}^{2} \cdot t_{5} \cdot t_{6}^{3} \cdot t_{9} \cdot t_{11}^{2}$$
Then, $$\pi^{\prime}=t_{4}^{2} \cdot t_{5} \cdot t_{6}^{3} \cdot t_{9} \cdot t_{11}^{2}.$$
By considering the standard $OGS$ elementary factorization for $\pi^{\prime}$, as it is defined in Definition \ref{canonical-factorization-def}, we get:

$$ \pi_{1}^{\prime} = t_{4}^{2} \cdot t_{5} \quad\quad \pi_{2}^{\prime} = t_{6}^{3} \cdot t_{9} \cdot t_{11}^{2}$$

 $$maj(\pi_{1}^{\prime}) < h_{1} , (3<4)\quad\quad maj(\pi_{2}^{\prime})= h_{1} , (6=6).$$

Hence, by Theorem \ref{s-nn-standard-ogs}
     $$w_{maj(\pi_{1}^{\prime})}\cdot \pi_{1}^{\prime}=w_{3} \cdot t_{4}^{2} \cdot t_{5}\in S^{\circ}_{n^{\prime}}\quad\quad \pi_{2}^{\prime}=t_{6}^{3} \cdot t_{9} \cdot t_{11}^{2}\in S^{\circ}_{n^{\prime}}.$$

Finally, by Theorem \ref{parabolic-factorization} the following holds:
$$\pi = (w_{3} \cdot t_{4}^{2} \cdot t_{5}) \cdot (t_{6}^{3} \cdot t_{9} \cdot t_{11}^{2})=w_{maj(\pi_{1}^{\prime})}\cdot \pi_{1}^{\prime}\cdot \pi_{2}^{\prime}\in S^{\circ}_{n^{\prime}}.$$
\end{example}

\section{The Coxeter length of elements of $D_{n}$}\label{len-of-Dn}

In this section we give a method to find the Coxeter length of elements of $D_n$ by using the generalized standard $OGS$ as it is defined in Definition \ref{ogs-dn}. We start with some lemmas which explain specific relations between the Coxeter generators $\{s_{1^{\prime}}, s_1, s_2\ldots s_{n-1}\} $ of $D_n$ and the elements\\ $w_k=s_k\cdot s_{k-1}\cdots s_1\cdot s_{1^{\prime}}\cdot s_2\cdots s_k$ for $1\leq k\leq n-1$, which are used  to find the Coxeter length of elements in $D_n$.

\begin{lemma}\label{ell-w}
For every $1\leq k\leq n-1$ let
\begin{equation}\label{equation-w}
 w_k=s_k\cdot s_{k-1}\cdots s_2\cdot s_1\cdot s_{1^{\prime}}\cdot s_2\cdots s_{k-1}\cdot s_k.   
\end{equation}
 as it is defined in Definition \ref{def w}. Then
\begin{itemize}
    \item $\ell(w_k)=2\cdot k$.
    \item $\ell(w_j\cdot w_k)=\ell(w_j)+\ell(w_k)=2\cdot (j+k)$\\ for $1\leq j, k\leq n-1$ and $j\neq k$.
\end{itemize}

\end{lemma}
\begin{proof}
The presentation of $w_k$ as it is presented in Equation \ref{equation-w}, is a presentation by the normal form of $D_n$ (See Definition \ref{D_n-normal-form}). Therefore, the presentation is reduced which contains $2\cdot k$ Coxeter generators. Hence, $\ell(w_k)=2\cdot k$.
Now, consider $\ell(w_j\cdot w_k)$. Since $j\neq k$ and by Proposition \ref{exchange-dn}, $w_j\cdot w_k=w_k\cdot w_j$ for every $1\leq j,k\leq n-1$, without loss of generality, we may assume $j<k$. 
Then by Definition \ref{D_n-normal-form}, the presentation:
$$w_j\cdot w_k=(s_j\cdot s_{j-1}\cdots s_2\cdot s_1\cdot s_{1^{\prime}}\cdot s_2\cdots s_{j-1}\cdot s_j)\cdot(s_k\cdot s_{k-1}\cdots s_2\cdot s_1\cdot s_{1^{\prime}}\cdot s_2\cdots s_{k-1}\cdot s_k)$$
is a presentation by the normal form. Hence, we get the result 
$$\ell(w_j\cdot w_k)=\ell(w_j)+\ell(w_k)=2\cdot (j+k),$$
for $1\leq j, k\leq n-1$ and $j\neq k$.
\end{proof}

\begin{lemma}\label{wl-sk}
Consider the group $D_n$ with the set of Coxeter generators $\{s_{1^{\prime}}, s_1, s_2\ldots s_{n-1}\} $ and for every $1\leq k\leq n-1$ let $w_k$ be the element of $D_n$ as it is defined in Definition \ref{def w} then the following relations holds.
\begin{itemize}
    \item $w_{L}\cdot s_{1} = s_{1^{\prime}} \cdot w_{L} \quad \text{for} \quad 2\leqslant L \leqslant n-1.$
    \item $w_{L}\cdot s_{1^{\prime}} = s_{1} \cdot w_{L} \quad \text{for} \quad 2\leqslant L \leqslant n-1.$
    \item $w_{L}\cdot s_{k} = s_{k} \cdot w_{L} \quad \text{for} \quad 2\leqslant k \leqslant L-1 \leqslant n-2.$
\end{itemize}
 
\end{lemma}
\begin{proof}
    \begin{itemize}
        \item Consider the first case, where $w_L \cdot s_1 = s_{1^{\prime}} \cdot w_L.$ $$ w_{L} \cdot s_{1} = s_{L} \cdots s_{2}\cdot s_{1}\cdot s_{1^{\prime}} \cdot s_{2}\cdots s_{L} \cdot s_{1}$$
          $$=  s_{L} \cdots s_{2}\cdot s_{1}\cdot s_{1^{\prime}} \cdot s_{2} \cdot s_{1}\cdot s_{3} \cdots s_{L}$$
          $$=  s_{L} \cdots s_{2}\cdot s_{1^{\prime}} \cdot s_{1} \cdot s_{2} \cdot s_{1}\cdot s_{3} \cdots s_{L}$$
          $$=  s_{L} \cdots s_{2}\cdot s_{1^{\prime}} \cdot s_{2} \cdot s_{1} \cdot s_{2} \cdots s_{L}$$
          $$=  s_{L} \cdots s_{3}\cdot s_{1^{\prime}}\cdot s_{2} \cdot s_{1^{\prime}} \cdot s_{1} \cdot s_{2} \cdots s_{L}$$
          $$= s_{1^{\prime}} \cdot  s_{L} \cdots  s_{2} \cdot s_{1^{\prime}} \cdot s_{1} \cdot s_{2} \cdots s_{L}$$
          $$= s_{1^{\prime}} \cdot w_{L} $$
          \item Consider the second case, where $w_L \cdot s_{1^{\prime}} = s_1 \cdot w_L.$ $$ w_{L} \cdot s_{1^{\prime}} = s_{L} \cdots s_{2}\cdot s_{1}\cdot s_{1^{\prime}} \cdot s_{2}\cdots s_{L} \cdot s_{1^{\prime}}$$
$$=  s_{L} \cdots s_{2}\cdot s_{1}\cdot s_{1^{\prime}} \cdot s_{2} \cdot s_{1^{\prime}} \cdots s_{L}$$
$$=  s_{L} \cdots s_{2}\cdot s_{1} \cdot s_{2} \cdot s_{1^{\prime}} \cdot s_{2} \cdots s_{L}$$
$$=  s_{L} \cdots s_{1}\cdot s_{2} \cdot s_{1} \cdot s_{1^{\prime}} \cdot s_{2} \cdots s_{L}$$
$$=  s_{1} \cdot s_{L}\cdots s_{2} \cdot  s_{1} \cdot s_{1^{\prime}} \cdot s_{2} \cdots s_{L}$$
$$= s_{1} \cdot w_{L} $$
         \item Consider the last case, where $w_L \cdot s_k = s_k \cdot w_L.$
         $$w_L \cdot s_k = s_L \cdots s_2 \cdot s_1 \cdot s_1^{\prime} \cdot s_2 \cdots s_{k-1} \cdot s_k \cdot s_{k+1}\cdots s_L \cdot s_k.$$
     $$=s_L \cdots s_k \cdots s_2\cdot s_1 \cdot s_1^{\prime} \cdot s_2 \cdots s_{k-1} \cdot s_k \cdot s_{k+1} \cdot s_k \cdots s_L.$$
      $$=s_L \cdots s_k \cdots s_2\cdot s_1 \cdot s_1^{\prime} \cdot s_2 \cdots s_{k-1} \cdot s_{k+1} \cdot s_{k} \cdot s_{k+1} \cdots s_L.$$
      $$=s_L \cdots s_{k+1}\cdot s_k \cdot s_{k+1} \cdot s_{k-1} \cdots s_2\cdot s_1 \cdot s_1^{\prime} \cdot s_2 \cdots s_{k-1} \cdot s_{k} \cdot s_{k+1} \cdots s_L.$$
      $$=s_L \cdots s_{k}\cdot s_{k+1} \cdot s_{k} \cdot s_{k-1} \cdots s_2\cdot s_1 \cdot s_1^{\prime} \cdot s_2 \cdots s_{k-1} \cdot s_{k} \cdot s_{k+1} \cdots s_L.$$
      $$=s_k \cdot s_L \cdots s_{k+1} \cdot s_{k} \cdot s_{k-1} \cdots s_2\cdot s_1 \cdot s_1^{\prime} \cdot s_2 \cdots s_{k-1} \cdot s_{k} \cdot s_{k+1} \cdots s_L.$$
      $$=s_{k} \cdot w_L.$$

    \end{itemize}
          
\end{proof}

\begin{lemma} \label{pi-j-r}
Let define  $\pi_{j,r}$ to be : $$\pi_{j,r} = s_{j} \cdots s_{j-r}$$
Then, $$w_{j} \cdot \pi_{j,r} = \pi_{j,r} \cdot w_{j-r-1}. $$ 

$$\ell(w_{j} \cdot \pi_{j,r}) = \ell( \pi_{j,r})+ 2\cdot (j-r-1) .$$

\end{lemma}
\begin{proof}
First, we recall $w_{j}$
$$w_{j}= s_{j} \cdot s_{j-1} \cdots s_{1}\cdot s_{1^{\prime}} \cdot s_{2} \cdots s_{j}$$

Then,
$$w_{j} \cdot \pi_{j,r} = s_{j} \cdots s_{1} \cdot s_{1^{\prime}} \cdots s_{j-r-1} \cdot s_{j-r} \cdots s_{j} \cdot s_{j} \cdots s_{j-r}.$$
$$= s_{j} \cdots s_{j-r} \cdot (s_{j-r-1} \cdots s_{1} \cdot s_{1^{\prime}} \cdots s_{j-r-1}).$$
Now, we look at $\pi_{j,r} \cdot w_{j-r-1}$
$$\pi_{j,r} \cdot w_{j-r-1} =  (s_{j} \cdots s_{j-r}) \cdot (s_{j-r-1} \cdots s_{1} \cdot s_{1^{\prime}} \cdots s_{j-r-1})$$ $$=norm(\pi_{j,r} \cdot w_{j-r-1})=norm(\pi_{j,r}\cdot norm( w_{j-r-1}).$$
Hence, we get:
$$\ell(w_{j} \cdot \pi_{j,r}) =  \ell( \pi_{j,r})+\ell( w_{j-r-1})=\ell( \pi_{j,r})+ 2\cdot (j-r-1) .$$
\end{proof}

Now, we develop an algorithm for the Coxeter length of an  element $\pi \in D_{n}$ such that $\pi^{\prime}$ is a standard $OGS$ elementary element by using its generalized standard $OGS$ presentation.

\begin{proposition}\label{len-pi}
  Let $\pi\in D_n$ such that $\pi^{\prime}= t_{k_{1}}^{i_{k_{1}}} \cdots t_{k_{m}}^{i_{k_{m}}}$ be a standard $OGS$ elementary element (i.e., $\sum_{j=1}^{m}i_{k_j}\leq k_1$) as it is defined in Definition \ref{elementary}. Assume,  the decomposition of $\pi$ in the form $\pi=\pi^{\bullet}\cdot \pi^{\circ}$ (as it is described in Theorem \ref{wt}) is as follow: $$\pi=w_L\cdot \pi^{\circ},$$ where $L$ is a positive integer, such that  $1\leq L\leq n-1$, and the element $w_L$ as it is defined in Definition \ref{def w}. Then the following holds:
 \begin{itemize}
     \item If  $\pi = w_{L} \cdot \pi^{\circ}$ and  $L< \sum_{j=1}^{m} i_{k_{j}}$\\
     Then, $$\ell(\pi)= 2L + \ell(\pi^{\prime}).$$
     \item If $\pi = w_{L} \cdot \pi^{\circ}$ and  $L= \sum_{j=1}^{m} i_{k_{j}}<k_1$\\
     Then,$$ \ell(\pi)=\ell(\pi^{\prime}).$$
     \item If  $\pi = w_{L} \cdot \pi^{\circ}$ and  $\sum_{j=1}^{m} i_{k_{j}}<L<k_1$\\
     Then,$$\ell(\pi)=2\cdot (L- \text{maj}(\pi^{\prime}))+ \ell(\pi^{\prime}).$$
     \item If  $\pi = w_{L} \cdot \pi^{\circ}$ and  $k_r\leq L<k_{r+1}$ for $1\leq r\leq m-1$\\
     Then $$\ell(\pi)=2\cdot (L- \sum_{j=r+1}^{m}i_{k_j})+ \ell(\pi^{\prime}).$$
     \item If  $\pi = w_{L} \cdot \pi^{\circ}$ and  $L\geq k_m$\\
     Then $$\ell(\pi)=2L + \ell(\pi^{\prime}).$$
 \end{itemize}
 \end{proposition}

 \begin{proof}

  Since $\pi^{\prime}$ is a standard $OGS$ elementary element (as it is defined in Definition \ref{elementary}), by \cite{S1} Theorem 28, $$norm(\pi^{\prime})=\prod_{u=\rho_1}^{k_1-1}\prod_{r=0}^{\rho_1-1}s_{u-r}\cdot \prod_{u=k_1}^{k_2-1}\prod_{r=0}^{\rho_2-1}s_{u-r}\cdot \prod_{u=k_2}^{k_3-1}\prod_{r=0}^{\rho_3-1}s_{u-r}\cdots \prod_{u=k_{m-1}}^{k_m-1}\prod_{r=0}^{\rho_m-1}s_{u-r},$$
 where, $\rho_{j}=\sum_{x=j}^{m}i_{k_{x}}$ for $1\leq j\leq m$.\\
    
     Notice, by Theorem \ref{wt}, the standard $OGS$ presentation of $\pi^{\circ}$ is the same tho the standard $OGS$ presentation of $\pi^{\prime}$, and $\pi^{\circ}$ and $\pi^{\prime}$ present the same element of $S_n$, by considering $\pi^{\circ}$ an element of $S_n$. Hence, $norm(\pi^{\circ})=norm(\pi^{\prime})$. Therefore, by considering the length of elements, we write $\ell(\pi^{\prime})$ instead of $\ell(\pi^{\circ})$.\\
     
     Notice, $norm(\pi^{\prime})$ has one of the  following form, which depends on either $maj(\pi^{\prime})<k_1$ or $maj(\pi^{\prime})=k_1$.
     \begin{itemize}
        \item If $maj(\pi^{\prime})<k_1$ then $$norm(\pi^{\prime})=(s_{maj(\pi^{\prime})}\cdot s_{maj(\pi^{\prime})-1}\cdots s_1)\cdot (s_{maj(\pi^{\prime})+1}\cdots)\cdots (s_{k_{m}-1}\cdots s_{k_{m}-i_{k_m}}).$$  
         \item If $maj(\pi^{\prime})=k_1$ then $$norm (\pi^{\prime})=(s_{k_1}\cdot s_{k_{1}-1}\cdots s_{k_{1}-\sum_{j=2}^{m}i_{k_j}})\cdot (s_{k_{1}+1}\cdots)\cdots  (s_{k_{m}-1}\cdots s_{k_{m}-i_{k_m}}).$$
         where $\sum_{j=2}^{m}i_{k_j}\geq 2$.
     \end{itemize}
    
    Then, we divide the proof into the 5 different cases of the value of $L$ compare to the value of $maj(\pi^{\prime})$ and $k_r$ for $1\leq r\leq n$.\\
    
    First consider the case $\pi = w_{L}\cdot \pi^{\circ} \quad  \text{where} \quad L<maj(\pi^{\prime})$: 
 $$w_L\cdot \pi^{\circ} = (s_{L}\cdot s_{L-1} \cdots s_2\cdot s_{1^{\prime}} \cdot s_{1} \cdots s_{L}) \cdot (s_{maj(\pi^{\prime})} \cdots)\cdots (s_{k_{m}-1}\cdots s_{k_{m}-i_{k_m}}) = norm(\pi).$$

  Therefore, in case $L < maj(\pi^{\prime})$ we conclude:
 $$\ell(\pi)= \ell(w_L) + \ell(\pi^{\prime})= 2L + \ell(\pi^{\prime}).$$
      
  Now, consider the case $\pi = w_{L}\cdot \pi^{\circ} \quad  \text{where} \quad L=maj(\pi^{\prime})<k_1.$
     Then, 
$$\pi= (s_{maj(\pi^{\prime})} \cdots s_2\cdot  s_{1^{\prime}} \cdot s_{1} \cdot s_2\cdots s_{maj(\pi^{\prime})}) \cdot (s_{maj(\pi^{\prime})} \cdots s_{1}) \cdot (s_{maj(\pi^{\prime})+1} \cdots  $$ $$ = (s_{maj(\pi^{\prime})} \cdots s_2\cdot  s_{1^{\prime}})\cdot (s_{maj(\pi^{\prime})+1} \cdots )\cdots (s_{k_{m}-1}\cdots s_{k_{m}-i_{k_m}})=norm(\pi).$$
Therefore, in case $L = maj(\pi^{\prime})<k_1$, we conclude that :
$$\ell(\pi)=\ell(\pi^{\prime}).$$

     Now, Consider the case $\pi = w_{L}\cdot \pi^{\prime} \quad  \text{where} \quad maj(\pi^{\circ}) < L < k_1$
     
 $$\pi = w_L\cdot \pi^{\circ} = (s_{L} \cdots s_{maj(\pi^{\prime})}\cdots  s_2\cdot s_{1^{\prime}} \cdot s_{1} \cdots s_{maj(\pi^{\prime})} \cdots s_{L}) \cdot( s_{maj(\pi^{\prime})} \cdots s_{1}) \cdots $$ $$\cdots  (s_{L-1}\cdots s_{L-maj(\pi^{\prime})}) \cdot (s_{L} \cdots s_{L-maj(\pi^{\prime})+1})\cdot (s_{L+1}\cdots $$
 By Lemma \ref{wl-sk}:
 $$w_{L} \cdot \pi^{\circ} = (s_{maj(\pi^{\prime})} \cdots s_{1^{\prime}}) \cdots (s_{L-1} \cdots s_{L-maj(\pi^{\prime})}) \cdot (s_{L} \cdots s_2\cdot s_{1^{\prime}} \cdot s_{1} \cdots s_{L-maj(\pi^{\prime})+1}\cdots  s_{L}) $$ $$ \cdot (s_{L} \cdots s_{L-maj(\pi^{\prime})+1})\cdot(s_{L+1}\cdots )\cdots(s_{k_{m}-1}\cdots s_{k_{m}-i_{k_m}}).$$ 
 By Lemma \ref{pi-j-r},  $$ w_L \cdot (s_{L} \cdots s_{L-maj(\pi^{\prime})+1})=(s_{L} \cdots s_{L-maj(\pi^{\prime})+1})\cdot w_{maj(\pi^{\prime})}.$$
 Hence, 
 $$w_{L} \cdot \pi^{\circ}=(s_{maj(\pi^{\prime})} \cdots s_{1^{\prime}}) \cdots (s_{L-1} \cdots s_{L-maj(\pi^{\prime})}) \cdot (s_{L} \cdots s_2\cdot s_{1^{\prime}} \cdot s_{1} \cdots s_{L-maj(\pi^{\prime})})\cdot (s_{L+1}\cdots )\cdots $$ $$\cdots(s_{k_{m}-1}\cdots s_{k_{m}-i_{k_m}})=norm(\pi).$$
  Therefore, in case $maj(\pi^{\prime})\leq L<k_{1}$ we conclude :
 $$\ell(\pi)=2\cdot (L- \text{maj}(\pi^{\prime}))+ \ell(\pi^{\prime}).$$

 Now, Consider the case $\pi = w_{L}\cdot \pi^{\circ} \quad  \text{where} \quad k_r \leq L < k_{r+1}$ \\ $ \text{for}\quad 1\leq r\leq m-1$
 
 $$\pi = w_L\cdot \pi^{\circ} = (s_{L} \cdots  s_2\cdot s_{1^{\prime}} \cdot s_{1} \cdots  s_{L}) \cdot( s_{maj(\pi^{\prime})} \cdots ) \cdots $$ $$\cdots  (s_{L-1}\cdots ) \cdot (s_{L} \cdots s_{L-\sum_{j=r+1}^m i_{k_j}+1})\cdot (s_{L+1}\cdots $$
 By Lemma \ref{wl-sk}:
 $$w_{L} \cdot \pi^{\circ} = (s_{maj(\pi^{\prime})} \cdots ) \cdots (s_{L-1} \cdots ) \cdot (s_{L} \cdots s_2\cdot s_{1^{\prime}} \cdot s_{1} \cdots  s_{L}) $$ $$ \cdot (s_{L} \cdots s_{L-\sum_{j=r+1}^m i_{k_j}+1})\cdot(s_{L+1}\cdots )\cdots(s_{k_{m}-1}\cdots s_{k_{m}-i_{k_m}}).$$ 
 By Lemma \ref{pi-j-r},  $$ w_L \cdot (s_{L} \cdots s_{L-\sum_{j=r+1}^m i_{k_j}+1})=(s_{L} \cdots s_{L-\sum_{j=r+1}^m i_{k_j}+1})\cdot w_{\sum_{j=r+1}^m i_{k_j}}.$$
 Hence, 
 $$=(s_{maj(\pi^{\prime})} \cdots ) \cdots (s_{L-1} \cdots ) \cdot (s_{L} \cdots s_2\cdot s_{1^{\prime}} \cdot s_{1} \cdots s_{L-\sum_{j=r+1}^m i_{k_j}})\cdot (s_{L+1}\cdots )\cdots $$ $$\cdots(s_{k_{m}-1}\cdots s_{k_{m}-i_{k_m}})=norm(\pi).$$
 Therefore, in case $k_r\leq L<k_{r+1}$ we conclude :
 $$\ell(\pi)=2\cdot (L-\sum_{j=r+1}^m i_{k_j} )+ \ell(\pi^{\prime}).$$
 
  Finally, consider the case $\pi = w_{L}\cdot \pi^{\circ} \quad  \text{where} \quad L\geq k_m$: 
 $$w_L\cdot \pi^{\circ} = (s_{L}\cdot s_{L-1} \cdots s_2\cdot s_{1^{\prime}} \cdot s_{1} \cdots s_{L}) \cdot (s_{maj(\pi^{\prime})} \cdots)\cdots (s_{k_{m}-1}\cdots s_{k_{m}-i_{k_m}}).$$
By Lemma \ref{wl-sk}:
$$w_L\cdot \pi^{\circ} = (s_{maj(\pi^{\prime})} \cdots)\cdots (s_{k_{m}-1}\cdots s_{k_{m}-i_{k_m}})\cdot(s_{L}\cdot s_{L-1} \cdots s_2\cdot s_{1^{\prime}} \cdot s_{1} \cdots s_{L}) .$$
  Therefore, in case $L \geq  k_m$ we conclude:
 $$\ell(\pi)= \ell(w_L) + \ell(\pi^{\prime})= 2L + \ell(\pi^{\prime}).$$
 
 \end{proof}

\begin{example}
 Consider $$\pi = w_{5}\cdot t_{7}^{2} \cdot t_{9}^{2} \cdot t_{12}^{3}$$ 
 $\bullet$ $L < \text{maj}(\pi^{\prime})$,\quad  $L=5<7=2+2+3=\text{maj}(\pi^{\prime})$.
 $$\pi = (s_5 \cdot s_4 \cdot s_3\cdot s_2\cdot s_1 \cdot s_1^{\prime} \cdot s_2\cdot s_3\cdot s_4\cdot s_5) \cdot (s_7 \cdot s_6 \cdot s_5 \cdot s_4 \cdot s_3)$$ $$ \cdot (s_8 \cdot s_7 \cdot s_6 \cdot s_5 \cdot s_4) \cdot (s_9 \cdot s_8 \cdot s_7) \cdot (s_{10} \cdot s_9 \cdot s_8) \cdot (s_{11} \cdot s_{10} \cdot s_9)=\text{norm}(\pi).$$
 Hence,
 $$\ell(\pi)= 2 \cdot 5 + 7\cdot 2 + 9\cdot 2 + 12 \cdot 3 - 7^{2} = 29.$$
 \end{example}
 \begin{example}
Consider $$\pi = w_{4} \cdot t_{5}^{2} \cdot t_{6}^{2}$$
 $\bullet$ $L = \text{maj}(\pi^{\prime})=4$.
 $$\pi = (s_4 \cdot s_3 \cdot s_2 \cdot  s_1^{\prime}\cdot s_1 \cdot s_2 \cdot s_3 \cdot s_4) \cdot (s_4 \cdot s_3 \cdot s_2 \cdot s_1)\cdot (s_5 \cdot s_4)$$
 $$ = (s_4 \cdot s_3 \cdot s_2 \cdot s_1^{\prime}) \cdot (s_5 \cdot s_4)=\text{norm}(\pi).$$
Hence, 
  $$\ell(\pi)= 5 \cdot 2 + 6\cdot 2 - 4^{2} =6 .$$
  \end{example}
  \begin{example}
  Consider $$\pi = w_{5} \cdot t_{7}^{2} \cdot t_{9}.$$
 $\bullet$ $\text{maj}(\pi^{\prime})<L<k_1$,\quad $\text{maj}(\pi^{\prime})=3<L=5<k_1=7$.
 $$\pi = (s_5 \cdot s_4 \cdots s_1 \cdot s_1^{\prime} \cdot s_2 \cdot s_3 \cdot s_4 \cdot s_5) \cdot (s_3 \cdot s_2 \cdot s_1)$$ $$  \cdot (s_4 \cdot s_3 \cdot s_2) \cdot (s_5 \cdot s_4 \cdot s_3) \cdot (s_6 \cdot s_5 \cdot s_4) \cdot s_7 \cdot s_8$$
 $$ =  (s_3 \cdot s_2 \cdot s_1) \cdot (s_4 \cdot s_3 \cdot s_2) \cdot (s_5 \cdot s_4 \cdot s_3\cdot s_2\cdot  s_1 \cdot s_1^{\prime} \cdot s_2 \cdot s_3 \cdot s_4 \cdot s_5) \cdot (s_5 \cdot s_4 \cdot s_3) \cdot (s_6 \cdot s_5 \cdot s_4) \cdot s_7 \cdot s_8$$
 $$ =  (s_3 \cdot s_2 \cdot s_1) \cdot (s_4 \cdot s_3 \cdot s_2) \cdot (s_5 \cdot s_4 \cdot s_3 \cdot s_2 \cdot s_1 \cdot s_1^{\prime} \cdot s_2)  \cdot (s_6 \cdot s_5 \cdot s_4) \cdot s_7 \cdot s_8=\text{norm}(\pi).$$
Hence,
 $$\ell(\pi)= 2 \cdot(5-3)+14 = 18 .$$
 \end{example}

Now, we consider the case of $\pi$, where $\pi^{\prime}$ is a standard OGS elementary element, but $\pi$ not necessarily of the form $w_{L}\cdot \pi^{\prime}$. In Theorem \ref{len-pi-wt} we consider the length of $\pi$ where $\pi$ is presented in the form $\pi=\pi^{\bullet}\cdot \pi^{\circ}$, as it is described in Theorem \ref{wt}, and in Theorem \ref{length-standard-general-dn} we consider the length of $\pi$ where $\pi$ is presented in the generalized standard $OGS$ presentation, as it is described in Theorem \ref{ogs-dn}.

\begin{theorem}\label{len-pi-wt}
Let $\pi\in D_n$ presented in the form $\pi=\pi^{\bullet}\cdot \pi^{\circ}$, as it is described in Theorem \ref{wt}. Assume 
$$\pi^{\bullet}=w_{L_1}\cdot w_{L_2}\cdots w_{L_u}, $$ for some positive integer $u$, where for $1\leq j\leq u$, ~$w_{L_j}$ as it is defined in Definition \ref{def w}.
Assume $\pi^{\circ}$ is a standard $OGS$ elementary element (as it is defined in Definition \ref{elementary}) by considering it as an element of $S_n$, such that the  standard $OGS$ presentation of $\pi^{\circ}$ as follow:  $$\pi^{\circ}=t_{k_1}^{i_{k_1}}\cdot t_{k_2}^{i_{k_2}}\cdots t_{k_m}^{i_{k_m}}.$$
Let $\pi^{\prime}$ be an element of $S_n$ as it is defined in Definition \ref{hom-dn-sn} (Notice, by Theorem \ref{wt}, the presentation of $\pi^{\prime}$  and $\pi^{\circ}$ have the same standard $OGS$ presentation, and present the same element of $S_n$, by considering $\pi^{\circ}$ as an element of $S_n$).
For every $1\leq j\leq u$, define $\varrho_{L_j}(\pi^{\prime})$ to be
$$\varrho_{L_j}(\pi^{\prime})=\begin{cases}
L_j & \text{if}\quad L_j<maj(\pi^{\prime})\quad \text{or}\quad Lj\geq k_m \\ \\
L_j-maj(\pi^{\prime}) & \text{if}\quad maj(\pi^{\prime})\leq L_j<k_1 \\ \\
L_j-\sum_{x=q}^m i_{k_x}\quad \text{for}\quad 2\leq q\leq m & \text{if}\quad k_{x-1}\leq L_j<k_x.
\end{cases}$$
Then $$\ell(\pi)=\ell(\pi^{\prime})+2\cdot\sum_{j=1}^{u} \varrho_{L_j}.$$

\end{theorem}

\begin{proof}

Consider $\pi=\pi^{\bullet}\cdot \pi^{\circ}$, where  $$\pi^{\bullet}=w_{L_1}\cdot w_{L_2}\cdots w_{L_u},$$ for some positive integer $u$, where for $1\leq j\leq u$, ~$w_{L_j}$ as it is defined in Definition \ref{def w}. Since by Proposition \ref{exchange-dn}, $w_{L_p}\cdot w_{L_q}=w_{L_q}\cdot w_{L_p}$ for every $1\leq p, q\leq u$, we may assume $$L_1<L_2< \ldots < L_u.$$

Hence, by Definition \ref{def w}, the presentation of $\pi^{\bullet}$ by Coxeter generators

$$\pi^{\bullet}=w_{L_1}\cdot w_{L_2}\cdots w_{L_u}$$ $$=(s_{L_1}\cdot s_{L_1-1}\cdots s_{2}\cdot s_{1}\cdot s_{1^{\prime}}\cdot s_{2}\cdots s_{L_1})\cdot (s_{L_2}\cdot s_{L_2-1}\cdots s_{2}\cdot s_{1}\cdot s_{1^{\prime}}\cdot s_{2}\cdots s_{L_2})\cdots $$ $$\cdots  (s_{L_u}\cdot s_{L_u-1}\cdots s_{2}\cdot s_{1}\cdot s_{1^{\prime}}\cdot s_{2}\cdots s_{L_u}).$$

Assume $\pi^{\circ}$ is a standard $OGS$ elementary element, by considering it as an element of $S_n$, with the following standard $OGS$ presentation as it is described in Theorem \ref{canonical-sn}.
$$\pi^{\circ}=t_{k_1}^{i_{k_1}}\cdot t_{k_2}^{i_{k_2}}\cdots t_{k_m}^{i_{k_m}}.$$
Then, by \cite {S1}, Theorem 28:

$$norm(\pi^{\circ})=\prod_{u=\rho_1}^{k_1-1}\prod_{r=0}^{\rho_1-1}s_{u-r}\cdot \prod_{u=k_1}^{k_2-1}\prod_{r=0}^{\rho_2-1}s_{u-r}\cdot \prod_{u=k_2}^{k_3-1}\prod_{r=0}^{\rho_3-1}s_{u-r}\cdots \prod_{u=k_{m-1}}^{k_m-1}\prod_{r=0}^{\rho_m-1}s_{u-r},$$
 where, $\rho_{j}=\sum_{x=j}^{m}i_{k_{x}}$ for $1\leq j\leq m$.\\

Now, we consider $norm(\pi)=norm(\pi^{\bullet}\cdot \pi^{\circ})$ in the following way.

First, consider $norm(w_{L_u}\cdot \pi^{\circ})$. Since, by Theorem \ref{wt}, the standard $OGS$ presentation of $\pi^{\circ}$ is same to the standard $OGS$ presentation of $\pi^{\prime}$, we can consider $norm(w_{L_u}\cdot \pi^{\circ})$ as it is described in the proof of Proposition \ref{len-pi}. Hence, by Proposition \ref{len-pi}
$$\ell(w_{L_u}\cdot \pi^{\circ})=2\cdot \dot{\varrho}_{L_u}(\pi^{\prime})+\ell(\pi^{\prime}),$$
where,
for every $1\leq j\leq u$, 
$$\dot{\varrho}_{L_j}(\pi^{\prime})=\begin{cases}
L_j & \text{if}\quad L_j<maj(\pi^{\prime})\quad \text{or}\quad Lj\geq k_m \\ \\
L_j-maj(\pi^{\prime}) & \text{if}\quad maj(\pi^{\prime})\leq L_j<k_1 \\ \\
L_j-\sum_{x=q}^m i_{k_x}\quad \text{for}\quad 2\leq q\leq m & \text{if}\quad k_{x-1}\leq L_j<k_x.
\end{cases}$$
Now, we consider $w_{L_{u-1}}\cdot norm(w_{L_u}\cdot \pi^{\circ})$.
Since, $L_{u-1}<L_u$, for finding $norm(w_{L_{u-1}}\cdot norm(w_{L_u}\cdot \pi^{\circ}))$, we can apply the same algorithm as it is described in Proposition \ref{len-pi} for finding $norm(w_L\cdot \pi^{\prime})$ and $\ell(w_L\cdot \pi^{\prime})$, and we get
$$\ell(w_{L_{u-1}}\cdot (w_{L_u}\cdot \pi^{\circ})=2\cdot \dot{\varrho}_{L_{u-1}}(\pi^{\prime})+\ell(w_{L_u}\cdot \pi^{\circ})=2\cdot \big(\dot{\varrho}_{L_{u-1}}(\pi^{\prime})+\dot{\varrho}_{L_{u}}(\pi^{\prime})\big)+\ell(\pi^{\prime}).$$

Now, assume by induction on $j$, 
$$\ell(w_{L_{u-j}}\cdot w_{L_{u-(j-1)}}\cdots w_{L_{u}}\cdot \pi^{\circ})=2\cdot\sum_{x=0}^j \dot{\varrho}_{L_{u-x}}(\pi^{\prime})+\ell(\pi^{\prime}).$$

Now, we consider $w_{L_{u-(j+1)}}\cdot norm(w_{L_{u-j}}\cdot w_{L_{u-(j-1)}}\cdots w_{L_{u}}\cdot \pi^{\circ})$.
Since, $L_{u-(j+1)}<L_{u-j}$, for finding  $$norm(w_{L_{u-(j+1)}}\cdot norm(w_{L_{u-j}}\cdot w_{L_{u-(j-1)}}\cdots w_{L_{u}}\cdot \pi^{\circ})),$$ we can apply the same algorithm as it is described in Proposition \ref{len-pi} for finding $norm(w_L\cdot \pi^{\circ})$ and $\ell(w_L\cdot \pi^{\circ})$, and we get
$$\ell(w_{L_{u-(j+1)}}\cdot (w_{L_{u-j}}\cdot w_{L_{u-(j-1)}}\cdots w_{L_{u}}\cdot \pi^{\circ}))$$ $$=2\cdot \dot{\varrho}_{L_{u-(j+1)}}(\pi^{\prime})+\ell((w_{L_{u-j}}\cdot w_{L_{u-(j-1)}}\cdots w_{L_{u}}\cdot \pi^{\circ})=2\cdot \sum_{x=0}^{j+1} \dot{\varrho}_{L_{u-x}}(\pi^{\prime})+\ell(\pi^{\circ}).$$

Hence, the induction assumption holds for every $0\leq u-1$, and we have:

$$\ell(\pi)=\ell(w_{L_1}\cdot w_{L_2}\cdots w_{L_u}\cdot \pi^{\circ})=\ell(\pi^{\circ})+2\cdot \sum_{j=1}^{u} \dot{\varrho}_{L_j}(\pi^{\prime}).$$

\end{proof}

\begin{theorem}\label{length-standard-general-dn}
Let $\pi$ be an element of $D_n$, such that $\pi^{\prime}$ is a standard $OGS$ elementary element (defined in Definition \ref{elementary}). Consider the presentation of $\pi$ as it is presented in Corollary \ref{twtw}, i.e.,$\pi$ is presented by the generalized standard $OGS$ as follow:
  $$\pi= \pi_{1}^{\circ}\cdot \pi_{1}^{\bullet}\cdot \pi_{2}^{\circ}\cdot \pi_{2}^{\bullet}\cdots \pi_{\mu-1}^{\circ}\cdot \pi_{\mu-1}^{\bullet}\cdot \pi_{\mu}^{\circ} $$ $$= \pi_{1}^{\circ}\cdot w_{L_{1_1}} \cdots w_{L_{1_{\nu_{1}}}} \cdot t_{k_{r_{1}+1}}^{i_{k_{r_{1}+1}}} \cdots t_{k_{r_{2}}}^{i_{k_{r_{2}}}}\cdot w_{L_{2_1}} \cdots w_{L_{2_{\nu_{2}}}} \cdot t_{k_{r_{2}+1}}^{i_{k_{r_{2}}+1}} \cdots $$ $$\cdots  t_{k_{r_{\mu-1}}}^{i_{k_{r_{\mu-1}}}} \cdot  w_{L_{{\mu-1}_1}} \cdots w_{L_{{\mu-1}_{\nu_{\mu-1}}}} \cdot \pi_{\mu}^{\circ},$$

where, 
$$\text{either}\quad \pi_{1}^{\circ}=t_{k_{1}}^{i_{k_{1}}} \cdots t_{k_{r_{1}}}^{i_{k_{r_{1}}}}\quad \text{or}\quad \pi_{1}^{\circ}=1,$$

and, 
$$\text{either}\quad  \pi_{\mu}^{\circ}=t_{k_{r_{\mu-1}+1}}^{i_{k_{r_{\mu-1}+1}}} \cdots t_{k_{r_{\mu}}}^{i_{k_{r_{\mu}}}}\quad  \text{or}\quad \pi_{\mu}^{\circ}=1.$$

Then,
$$\ell(\pi) =\begin{cases} \sum_{u=1}^{\mu-1} \big(maj_{u}(\pi)\cdot  \big((-1)^{\nu_{u}}+1\big)+2\cdot\sum_{j=1}^{\nu_{u}}\varrho_{L_{u_j}}(\pi)\big) + \ell(\pi^{\prime}) & \text{if}\quad \pi_{\mu}^{\circ}\neq 1 \\ \\
\sum_{u=1}^{\mu-2} \big(maj_{u}(\pi)\cdot \big((-1)^{\nu_{u}}+1\big)+2\cdot\sum_{j=1}^{\nu_{u}}\varrho_{L_{u_j}}(\pi)\big) + \ell(\pi^{\prime}) & \text{if}\quad \pi_{\mu}^{\circ} = 1.
\end{cases}$$
where, by Definition \ref{maj-alpha},

$$maj_{u}(\pi) =  \sum_{j=1}^{r_{u}} i_{k_{j}}\quad \rho_{u}(\pi)=\sum_{j=r_{u}+1}^{r_{\mu}} i_{k_{j}},$$
 $$\varrho_{L_{u_j}}(\pi^{\circ})=\begin{cases}
  L_{u_j}-\rho_{u}(\pi) & \text{if}\quad L_{u_j}\geq maj(\pi^{\prime}) \\
  L_{u_j} & \text{if}\quad L_{u_j} < maj(\pi^{\prime})
  \end{cases}$$

\end{theorem}
\begin{proof}
Consider the presentation of $\pi$ as it is presented in Corollary \ref{twtw}. 
$$\pi= \pi_{1}^{\circ}\cdot \pi_{1}^{\bullet}\cdot \pi_{2}^{\circ}\cdot \pi_{2}^{\bullet}\cdots \pi_{\mu-1}^{\circ}\cdot \pi_{\mu-1}^{\bullet}\cdot \pi_{\mu}^{\circ} $$ $$= \pi_{1}^{\circ} \cdot w_{L_{1_1}} \cdots w_{L_{1_{\nu_{1}}}} \cdot t_{k_{r_{1}+1}}^{i_{k_{r_{1}+1}}} \cdots t_{k_{r_{2}}}^{i_{k_{r_{2}}}}\cdot w_{L_{2_1}} \cdots w_{L_{2_{\nu_{2}}}} \cdot t_{k_{r_{2}+1}}^{i_{k_{r_{2}}+1}} \cdots $$ $$\cdots  t_{k_{r_{\mu-1}}}^{i_{k_{r_{\mu-1}}}} \cdot  w_{L_{{\mu-1}_1}} \cdots w_{L_{{\mu-1}_{\nu_{\mu-1}}}} \cdot \pi_{\mu}^{\circ}.$$

where, 
$$\text{either}\quad \pi_{1}^{\circ}=t_{k_{1}}^{i_{k_{1}}} \cdots t_{k_{r_{1}}}^{i_{k_{r_{1}}}}\quad \text{or}\quad \pi_{1}^{\circ}=1,$$

and
$$\text{either}\quad  \pi_{\mu}^{\circ}=t_{k_{r_{\mu-1}+1}}^{i_{k_{r_{\mu-1}+1}}} \cdots t_{k_{r_{\mu}}}^{i_{k_{r_{\mu}}}}\quad  \text{or}\quad \pi_{\mu}^{\circ}=1.$$

Now, we recall the definitions of  $maj_{\alpha}(\pi)$, $\rho_{\alpha}(\pi)$ and  $\varrho_{L_{\alpha_j}}(\pi)$ for every $1\leq \alpha\leq \mu$ and $1\leq j\leq \nu_{\alpha}$, as it is defined in Definition \ref{maj-alpha} :
$$maj_{\alpha}(\pi) =  \sum_{j=1}^{r_{\alpha}} i_{k_{j}}\quad\quad \rho_{\alpha}(\pi) = maj(\pi^{\prime})-maj_{\alpha}(\pi)= \sum_{j=r_{\alpha}+1}^{r_{\mu}} i_{k_{j}}.$$

$$\varrho_{L_{u_j}}(\pi)=\begin{cases}
  L_{u_j}-\rho_{u}(\pi) & \text{if}\quad L_{u_j}\geq maj(\pi^{\prime}) \\
  L_{u_j} & \text{if}\quad L_{u_j} <  maj(\pi^{\prime})
  \end{cases}$$

By Theorem \ref{wt}, 

$$\pi =\prod_{j=1}^{\mu-1} w_{maj_{j}(\pi) ~|~ maj_{j}(\pi)<k_1}^{-0.5 \cdot (-1)^{\nu_{j}}+0.5}\cdot w_{L_{1_1}} \cdots w_{L_{1_{\nu_1}}} \cdot w_{L_{2_1}} \cdots w_{L_{2_{\nu_2}}}  \cdots w_{L_{{\mu-1}_1}} \cdots w_{L_{{\mu-1}_{\nu_{\mu-1}}}}  \cdot \pi^{\circ}.$$

First, notice by Theorem \ref{wt},
$$\pi^{\circ}=\pi_{1}^{\circ}\cdot \pi_{2}^{\circ}\cdots \pi_{\mu}^{\circ}
=\begin{cases}
t_{k_1}^{i_{k_1}}\cdot t_{k_2}^{i_{k_2}}\cdots t_{k_{r_{\mu-1}}}^{i_{k_{r_{\mu-1}}}} & \text{if}\quad \pi_{\mu}^{\circ}=1\\ 
t_{k_1}^{i_{k_1}}\cdot t_{k_2}^{i_{k_2}}\cdots t_{k_{r_{\mu-1}}}^{i_{k_{r_{\mu-1}}}}\cdot t_{k_{r_{\mu-1}+1}}^{i_{k_{r_{\mu-1}+1}}}\cdots t_{k_{r_{\mu}}}^{i_{k_{r_{\mu}}}}   & \text{if}\quad \pi_{\mu}^{\circ}\neq 1.
\end{cases}$$

Now, we consider  $L_{\alpha_j}$ for $1\leq \alpha\leq \mu-1$ and  $1\leq j\leq \nu_{\alpha}$.\\

By Corollary \ref{twtw}, 
\begin{equation}\label{L-j-k-r}
\begin{cases}
L_{\alpha_j}<k_1  & \text{if}\quad  \alpha=1\quad \text{and}\quad \pi_{1}^{\circ}=1\\
L_{\alpha_j}\geq k_{m-1}  & \text{if}\quad  \alpha=\mu-1\quad \text{and}\quad \pi_{\mu}^{\circ}=1\\
k_{r_{\alpha}}\leq L_{\alpha_j} < k_{r_{\alpha}+1} & \text{otherwise} 
\end{cases}
\end{equation}

Notice, by Theorem \ref{wt}, the standard $OGS$ presentation of $\pi^{\circ}$ by considering it as an element of $S_n$ is  same to the standard $OGS$ presentation of $\pi^{\prime}$. Hence, when we calculate length of $\pi$ or normal form of $\pi$ we often write $\pi^{\prime}$ instead of $\pi^{\circ}$.\\ 

Now, consider $\dot{\varrho}_{L_{\alpha_j}}(\pi^{\prime})$ for $1\leq \alpha\leq \mu-1$ and  $1\leq j\leq \nu_{\alpha}$, as it is defined in  Theorem \ref{len-pi-wt}.\\

 $$ \dot{\varrho}_{L_{\alpha_j}}(\pi^{\prime})=\begin{cases}
L_{\alpha_j} & \text{if}\quad L_{\alpha_j}<maj(\pi^{\prime})\quad \text{or}\quad L_{\alpha_j}\geq k_{\mu} \\ \\
L_{\alpha_j}-maj(\pi^{\prime}) & \text{if}\quad maj(\pi^{\prime})\leq L_{\alpha_j}<k_1 \\ \\
L_{\alpha_j}-\sum_{x=q}^m i_{k_x}\quad \text{for}\quad 2\leq q\leq \mu & \text{if}\quad k_{x-1}\leq L_{\alpha_j}<k_x.
\end{cases}$$

Hence, by Equation \ref{L-j-k-r}, and by the definition of $\varrho_{\alpha}(\pi)$ for $1\leq \alpha\leq \mu-1$, the following are satisfied: \\

For $1\leq \alpha\leq \mu-1$, and $1\leq j\leq \nu_{\alpha}$, 

\begin{equation}\label{L-j-rho-dot}
 \dot{\varrho}_{L_{\alpha_j}}(\pi^{\prime})=\varrho_{\alpha}(\pi).   
\end{equation}

Now, we consider $maj_{\alpha}(\pi)$ for $1\leq \alpha\leq \mu-1$.\\

Notice, by the definition of $maj_{\alpha}(\pi)$ for $1\leq \alpha\leq \mu-1$, the following holds:
$$maj_{\alpha}(\pi)=\prod_{j=1}^{\alpha}maj(\pi_j^{\circ})\leq \prod_{j=1}^{\mu}maj(\pi_j^{\circ})=maj(\pi^{\prime}),$$
by considering $\pi^{\circ}$ as an element of $S_n$.

where,
$$maj_{\alpha}(\pi)= maj(\pi^{\prime}) $$
if and only if 
$$\alpha=\mu-1\quad \text{and}\quad \pi_{\mu}^{\circ}=1.$$

Therefore, in case:
$$\pi_{\mu}^{\circ}\neq 1\quad \text{or}\quad 1\leq\alpha\leq \mu-2$$
we have $$maj_{\alpha}(\pi)< maj(\pi^{\prime}).$$

Hence, by the definition of $\dot{\varrho}_{maj_{\alpha}(\pi)}(\pi^{\prime})$, for every $\leq \alpha\leq \mu-1$, such that $maj_{\alpha}(\pi)<k_1$, the following holds:

\begin{equation}\label{maj-rho-dot}
\dot{\varrho}_{maj_{\alpha}(\pi)}(\pi^{\prime})=\begin{cases}
maj_{\alpha}(\pi) & \text{if}\quad  \pi_{\mu}^{\circ}\neq 1\quad \text{or}\quad 1\leq\alpha\leq \mu-2 \\
0 & \text{if}\quad \alpha=\mu-1\quad \text{and}\quad \pi_{\mu}^{\circ}=1.
\end{cases}    
\end{equation}

Hence, by Theorem \ref{len-pi-wt}, by using Equations \ref{L-j-rho-dot} and  \ref{maj-rho-dot} we get the length formula:

$$\ell(\pi) =$$ $$ =\ell(\prod_{j=1}^{\mu-1} w_{maj_{j}(\pi) ~|~ maj_{j}(\pi)<k_1}^{-0.5 \cdot (-1)^{\nu_{j}}+0.5}\cdot w_{L_{1_1}} \cdots w_{L_{1_{\nu_1}}} \cdot w_{L_{2_1}} \cdots w_{L_{2_{\nu_2}}}  \cdots w_{L_{{\mu-1}_1}} \cdots w_{L_{{\mu-1}_{\nu_{\mu-1}}}}  \cdot \pi^{\circ}) $$   $$=\begin{cases} \sum_{u=1}^{\mu-1} \big(maj_{u}(\pi)\cdot  \big((-1)^{\nu_{u}}+1\big)+2\cdot\sum_{j=1}^{\nu_{u}}\varrho_{L_{u_j}}(\pi)\big) + \ell(\pi^{\prime}) & \text{if}\quad \pi_{\mu}^{\circ}\neq 1 \\ \\
\sum_{u=1}^{\mu-2} \big(maj_{u}(\pi)\cdot \big((-1)^{\nu_{u}}+1\big)+2\cdot\sum_{j=1}^{\nu_{u}}\varrho_{L_{u_j}}(\pi)\big) + \ell(\pi^{\prime}) & \text{if}\quad \pi_{\mu}^{\circ} = 1.
\end{cases}$$

\end{proof}
\begin{example}\label{standsrd-length-example}
Consider
$$\pi = t_{10}^{4} \cdot w_{11} \cdot t_{12}^{2}\cdot w_{13} \cdot w_{14} \cdot w_{15} \cdot t_{16}^{3} \cdot t_{17}.$$

Notice,
$$\pi=\pi_1^{\circ}\cdot\pi_1^{\bullet}\cdot \pi_2^{\circ}\cdot\pi_2^{\bullet}\cdot\pi_3^{\circ},$$
where,
$$\pi_1^{\circ}=t_{10}^{4}\quad \pi_1^{\bullet}=w_{11}\quad \pi_2^{\circ}=t_{12}^{2}\quad \pi_2^{\bullet}= w_{13} \cdot w_{14} \cdot w_{15}\quad \pi_3^{\circ}=t_{16}^{3} \cdot t_{17}. $$

$$maj_1(\pi)=maj(\pi_1^{\circ})=4\quad maj_2(\pi)=maj(\pi_1^{\circ}\cdot\pi_2^{\circ})=4+2=6.$$
$$maj(\pi^{\prime})= maj(\pi_1^{\circ}\cdot\pi_2^{\circ} \cdot \pi_3^{\circ} )= 4+2+4=10.$$

$$\rho_{1}(\pi) = maj(\pi^{\prime})-maj_{1}(\pi)=10 - 4=6 \quad \rho_{2}(\pi) = maj(\pi^{\prime})-maj_{2}(\pi)=10 - 6=4. $$

Hence, by Theorem \ref{length-standard-general-dn}
$$\ell(\pi) = 2 \cdot maj_1(\pi) + 2 \cdot maj_2(\pi) + 2\cdot(11-\rho_1(\pi)) + 2\cdot (13-\rho_2(\pi)) + 2\cdot (14-\rho_2(\pi)) + 2\cdot$$ $$(15-\rho_2(\pi)) + 10 \cdot 4 + 12 \cdot 2 + 16 \cdot 3 + 17 -10^{2}$$ $$=2\cdot 4 + 2\cdot 6 + 2\cdot(11-6)+2\cdot(13-4)+2\cdot(14-4)+2\cdot(15-4)+$$ $$+10\cdot 4+12\cdot 2+16\cdot 3+17-10^2=90+29=119.$$
\\

Now, consider $\ell(\pi)$ by looking at the presentation of $\pi$ in terms of Coxeter generators:\\
First, notice by Theorem \ref{wt}

$$\pi= w_{4}\cdot w_{6} \cdot w_{11} \cdot w_{13} \cdot w_{14} \cdot w_{15} \cdot t_{10}^{4} \cdot t_{12}^{2}\cdot t_{16}^{3} \cdot t_{17}.$$
Then, the presentation of $\pi$ in terms of Coxeter generators as follow:

$$\pi = (s_{4}\cdot s_{3}\cdot s_{2}\cdot s_{1}\cdot s_{1^{\prime}}\cdot s_{2}\cdot s_3\cdot s_{4})\cdot (s_{6}\cdot s_{5}\cdots s_{1}\cdot s_{1^{\prime}}\cdot s_{2}\cdots s_{6})\cdot (s_{11}\cdot s_{10}\cdots s_{1}\cdot s_{1^{\prime}}\cdot s_{2}\cdots  s_{11}) $$ $$\cdot (s_{13}\cdot s_{12}\cdots s_{1}\cdot s_{1^{\prime}}\cdot
s_{2}\cdots s_{13})\cdot (s_{14}\cdot s_{13}\cdots s_{1}\cdot s_{1^{\prime}}\cdot s_{2}\cdots s_{14})\cdot (s_{15}\cdots s_{1}\cdot  s_{1^{\prime}}\cdot s_{2}\cdots s_{15})$$ $$\cdot (s_{10}\cdot s_{9}\cdot s_{8}\cdot s_{7} \cdot s_{6} \cdot s_{5})\cdot (s_{11}\cdot s_{10}\cdot s_{9}\cdot s_{8}\cdot s_{7}\cdot s_{6})\cdot (s_{12}\cdot s_{11}\cdot s_{10} \cdot  s_{9})\cdot (s_{13}\cdot s_{12}\cdot s_{11}\cdot s_{10})$$ $$\cdot (s_{14}\cdot s_{13}\cdot s_{12}\cdot s_{11})\cdot (s_{15}\cdot s_{14}\cdot s_{13}\cdot s_{12})\cdot s_{16}. $$

By Lemma \ref{wl-sk}, we get  :
$$\pi = (s_{4}\cdot s_{3}\cdot s_{2}\cdot s_{1}\cdot s_{1^{\prime}}\cdot s_{2}\cdot s_{3}\cdot s_{4})\cdot (s_{6}\cdot s_{5}\cdots s_{1}\cdot s_{1^{\prime}}\cdot s_{2}\cdots s_{6})\cdot  (s_{10}\cdot s_{9}\cdot s_{8}\cdot s_{7}\cdot s_{6}\cdot s_{5})\cdot $$ $$ (s_{11}\cdot s_{10}\cdots s_{1}\cdot s_{1^{\prime}}\cdot s_{2}\cdots  s_{11})\cdot (s_{11}\cdot s_{10}\cdots s_{6})\cdot (s_{12}\cdot  s_{11}\cdot s_{10}\cdot  s_{9})\cdot  (s_{13}\cdot s_{12}\cdots s_{1}\cdot s_{1^{\prime}}\cdot s_{2}\cdots s_{13})\cdot $$ $$  (s_{13}\cdot s_{12}\cdot s_{11}\cdot s_{10})\cdot (s_{14}\cdot s_{13}\cdots s_{1}\cdot s_{1^{\prime}}\cdot s_{2}\cdots s_{14})\cdot (s_{14}\cdot s_{13}\cdot s_{12}\cdot s_{11})\cdot $$ $$   (s_{15}\cdots s_{1}\cdot  s_{1^{\prime}}\cdot s_{2}\cdots s_{15})\cdot  (s_{15}\cdot s_{14}\cdot s_{13}\cdot  s_{12})\cdot s_{16} $$
$$= (s_{4}\cdot s_{3}\cdot s_{2}\cdot s_{1}\cdot s_{1^{\prime}}\cdot s_{2}\cdot s_{3}\cdot s_{4})\cdot (s_{6}\cdot s_{5}\cdots s_{1}\cdot s_{1^{\prime}}\cdot s_{2}\cdots s_{6})\cdot  (s_{10}\cdot s_{9}\cdot s_{8}\cdot s_{7}\cdot s_{6}\cdot s_{5})\cdot (s_{11}\cdot s_{10}\cdots s_{1}\cdot  s_{1^{\prime}}\cdot s_{2}\cdots s_{5})$$ $$ \cdot (s_{12}\cdot  s_{11}\cdot s_{10} \cdot  s_{9})\cdot  (s_{13}\cdot s_{12}\cdots s_{1}\cdot s_{1^{\prime}}\cdot s_{2}\cdots s_{9}) \cdot  (s_{14}\cdot s_{13}\cdots s_{1}\cdot s_{1^{\prime}}\cdot s_{2}\cdots s_{10})$$ $$ \cdot  (s_{15}\cdots s_{1}\cdot  s_{1^{\prime}}\cdot s_{2}\cdots s_{11})\cdot s_{16}=\text{norm}(\pi).$$
Thus:

$$\ell(\pi) = 2 \cdot 4 + 2 \cdot 6 + 2\cdot(11-6) + 2\cdot (13-4) + 2\cdot (14-4) + 2\cdot (15-4)+$$ $$+ 10 \cdot 4 + 12 \cdot 2 + 16 \cdot 3 + 17 -(10^{2})=90+29=119.$$

\end{example}

In the next example, we present an algorithm to find the Coxeter length of an arbitrary element  $\pi$ of $D_n$, presented by the generalized standard $OGS$ of $D_n$, where $\pi^{\prime}$ is not necessarily a standard $OGS$ elementary element.

\begin{example}\label{general-length-example}
Consider $$\pi = t_{7}^{2}\cdot  t_{9}^{2} \cdot t_{12}^{4} \cdot t_{14}^{7} \cdot w_{15} \cdot t_{16}^{4} \cdot t_{19}^{10} \cdot w_{20} \cdot t_{22}^{5} \cdot t_{23}^{4}.$$

Therefore,

$$\pi^{\prime}=t_{7}^{2}\cdot  t_{9}^{2} \cdot t_{12}^{4} \cdot t_{14}^{7}\cdot t_{16}^{4} \cdot t_{19}^{10}\cdot t_{22}^{5} \cdot t_{23}^{4}.$$
Notice,
$$maj(\pi^{\prime})= 2+2+4+7+4+10+5+4=38 > k_1 =7.$$

Hence $\pi^{\prime}$ is not a standard $OGS$ elementary element, So, we consider a standard $OGS$ elementary factorization of $\pi^{\prime}$ as it is described in Definition \ref{canonical-factorization-def}:
$$\pi^{\prime} = t_{7}^{2} \cdot t_{9}^{2} \cdot t_{12}^{3} \quad | t_{12} \cdot t_{14}^{7} \cdot t_{16}^{4} \quad | t_{19}^{10} \cdot t_{22}^{5} \cdot t_{23}^{4}$$

Therefore,
$$\pi = t_{7}^{2} \cdot t_{9}^{2} \cdot t_{12}^{3} \quad | t_{12} \cdot t_{14}^{7} \cdot w_{15} \cdot t_{16}^{4} \quad | t_{19}^{10} \cdot w_{20} \cdot t_{22}^{5} \cdot t_{23}^{4}$$
Then by Remark \ref{wt-general}:
$$\dot{\pi}^{(1)} = t_{7}^{2} \cdot t_{9}^{2} \cdot t_{12}^{3} \quad \dot{\pi}^{(2)} =  t_{12} \cdot t_{14}^{7} \cdot w_{15} \cdot t_{16}^{4} \quad \dot{\pi}^{(3)} = t_{19}^{10} \cdot w_{20} \cdot t_{22}^{5} \cdot t_{23}^{4} $$
Notice,
$$maj\big[{\big(\pi^{\prime}\big)}^{(1)}\big]=7,$$
$$maj\big[{\big(\pi^{\prime}\big)}^{(2)}\big]= 12 \quad maj_1(\dot{\pi}^{(2)})= 8,$$ $$ maj\big[{\big(\pi^{\prime}\big)}^{(3)}\big] = 19 \quad maj_1(\dot{\pi}^{(3)})= 10.$$

Hence, by Theorem \ref{wt} on $\dot{\pi}^{(j)}$ for $1\leq j\leq 3$:

$$\pi = t_{7}^{2} \cdot t_{9}^{2} \cdot t_{12}^{3} \quad |w_{8} \cdot w_{15} \cdot t_{12} \cdot t_{14}^{7} \cdot t_{16}^{4} \quad |w_{10}\cdot w_{20} \cdot t_{19}^{10} \cdot  t_{22}^{5} \cdot t_{23}^{4}$$
By applying Proposition \ref{exchange-dn} on $t_{12}^{3} \cdot w_{8}$   and on  $t_{16}^{4} \cdot w_{10}$ we get :
$$\pi = t_{7}^{2} \cdot t_{9}^{2}\cdot w_3 \cdot w_{11} \cdot  t_{12}^{3} \quad |  w_{15} \cdot t_{12} \cdot t_{14}^{7} \cdot w_4\cdot w_{14} \cdot t_{16}^{4} \quad | w_{20} \cdot t_{19}^{10} \cdot  t_{22}^{5} \cdot t_{23}^{4}$$
By applying Proposition \ref{exchange-dn} on $t_{9}^{2} \cdot w_{3}$   and on  $t_{14}^{7} \cdot w_{4}$ we get :

$$\pi = t_{7}^{2}\cdot w_2 \cdot w_5  \cdot t_{9}^{2}\cdot w_{11} \cdot  t_{12}^{3} \quad |  w_{15} \cdot t_{12}  \cdot w_7 \cdot w_{11}\cdot t_{14}^{7}\cdot w_{14} \cdot t_{16}^{4} \quad | w_{20} \cdot t_{19}^{10} \cdot  t_{22}^{5} \cdot t_{23}^{4}$$

By applying Proposition \ref{exchange-dn} on $t_{7}^{2} \cdot w_{2}$   and on  $t_{12} \cdot w_{7}$ we get :

$$\pi = w_2 \cdot w_4\cdot t_{7}^{2}\cdot w_5  \cdot t_{9}^{2}\cdot w_{11} \cdot  t_{12}^{3} \quad | w_{15}\cdot w_1 \cdot w_8 \cdot t_{12}  \cdot w_{11}\cdot t_{14}^{7}\cdot w_{14} \cdot t_{16}^{4} \quad | w_{20} \cdot t_{19}^{10} \cdot  t_{22}^{5} \cdot t_{23}^{4}$$

By applying Proposition \ref{exchange-dn} on $t_{7}^{2} \cdot w_{5}$   and on  $t_{12} \cdot w_{11}$ we get :

$$\pi = w_2\cdot w_4\cdot w_2 \cdot t_{7}^{2}\cdot t_{9}^{2}\cdot w_{11} \cdot  t_{12}^{3} \quad | w_{15} \cdot w_1 \cdot w_8\cdot w_1 \cdot t_{12} \cdot t_{14}^{7}\cdot w_{14} \cdot t_{16}^{4} \quad | w_{20} \cdot t_{19}^{10} \cdot  t_{22}^{5} \cdot t_{23}^{4}$$

By applying Proposition \ref{exchange-dn} on $w_4\cdot w_2$ and on $w_8\cdot w_1$ we get:

$$\pi = w_2\cdot w_2\cdot w_4 \cdot t_{7}^{2}\cdot t_{9}^{2}\cdot w_{11} \cdot  t_{12}^{3} \quad | w_{15} \cdot w_1 \cdot w_1\cdot w_8 \cdot t_{12} \cdot t_{14}^{7}\cdot w_{14} \cdot t_{16}^{4} \quad | w_{20} \cdot t_{19}^{10} \cdot  t_{22}^{5} \cdot t_{23}^{4}$$

Since by applying Lemma \ref{order-w},  $w_2^2=w_1^2=1$:

$$\pi =  w_4 \cdot t_{7}^{2}\cdot t_{9}^{2}\cdot w_{11} \cdot  t_{12}^{3} \quad | w_{15}\cdot w_8 \cdot t_{12} \cdot t_{14}^{7}\cdot w_{14} \cdot t_{16}^{4} \quad | w_{20} \cdot t_{19}^{10} \cdot  t_{22}^{5} \cdot t_{23}^{4}$$

Then, by applying Lemma  \ref{tw-standard} on $t_{7}^{2}\cdot t_{9}^{2}\cdot w_{11}$ and on $t_{12} \cdot t_{14}^{7}\cdot w_{14}$:

$$\pi = w_4\cdot w_4\cdot w_{11} \cdot t_{7}^{2} \cdot t_{9}^{2}\cdot   t_{12}^{3} \quad | w_{15} \cdot w_{8} \cdot w_{8}\cdot w_{14} \cdot t_{12} \cdot t_{14}^{7} \cdot  t_{16}^{4} \quad | w_{20} \cdot t_{19}^{10} \cdot  t_{22}^{5} \cdot t_{23}^{4}$$

By Lemma \ref{order-w}, $w_4^2=w_8^2=1$:
$$\pi = w_{11} \cdot t_{7}^{2} \cdot t_{9}^{2}\cdot   t_{12}^{3} \quad | w_{15} \cdot w_{14} \cdot t_{12} \cdot t_{14}^{7} \cdot  t_{16}^{4} \quad | w_{20} \cdot t_{19}^{10} \cdot  t_{22}^{5} \cdot t_{23}^{4}$$

Then, by applying Proposition \ref{exchange-dn} on $w_{15}\cdot w_{14}$ we get:

\begin{equation}\label{factorization-dn}
\pi = w_{11} \cdot t_{7}^{2} \cdot t_{9}^{2}\cdot   t_{12}^{3} \quad | w_{14} \cdot w_{15} \cdot t_{12} \cdot t_{14}^{7} \cdot  t_{16}^{4} \quad | w_{20} \cdot t_{19}^{10} \cdot  t_{22}^{5} \cdot t_{23}^{4}
\end{equation}

Then, $$\pi=\pi^{(1)}\cdot \pi^{(2)}\cdot\pi^{(3)} $$
such that:

$$\pi^{(1)}=w_{11}\cdot t_{7}^{2}\cdot t_{9}^{2}\cdot  t_{12}^{3}\quad \pi^{(2)}= w_{14} \cdot w_{15} \cdot t_{12} \cdot t_{14}^{7} \cdot  t_{16}^{4}\quad \pi^{(3)}= w_{20} \cdot t_{19}^{10} \cdot  t_{22}^{5} \cdot t_{23}^{4}$$

Now, consider $\pi^{(1)}$.
$$\pi^{(1)}=w_{11}\cdot t_{7}^{2}\cdot t_{9}^{2}\cdot  t_{12}^{3},$$
where $t_{7}^{2}\cdot t_{9}^{2}\cdot  t_{12}^{3}$ is a standard $OGS$ elementary element, with \\ $k_1=7, \quad k_2=9, \quad k_3=12$, and $i_{k_1}=2, \quad i_{k_2}=2, \quad i_{k_3}=3$, as it is defined in Definition \ref{elementary}.\\

Since, $9<11<12$, we consider Proposition \ref{len-pi} on $w_{11}\cdot (t_{7}^{2}\cdot t_{9}^{2}\cdot  t_{12}^{3})$, for the case $k_2\leq L<k_3$.

Then, we get:

$$norm(\pi^{(1)})=(s_7\cdot s_6\cdot s_5\cdot s_4\cdot s_3)\cdot (s_8\cdot s_7\cdot s_6\cdot s_5\cdot s_4)\cdot (s_9\cdot s_8\cdot s_7)\cdot (s_{10}\cdot s_9\cdot s_8)$$ $$\cdot (s_{11}\cdot s_{10}\cdots s_2\cdot s_{1^{\prime}}\cdot s_1\cdot s_2\cdots s_8).$$

The length of $\pi^{(1)}$:
$$\ell(\pi^{(1)}) = \ell(t_{7}^{2}\cdot t_{9}^{2}\cdot  t_{12}^{3}) + 2\cdot (11-3).$$

Now, consider $\pi^{(2)}$.
$$\pi^{(2)}=w_{14}\cdot w_{15}\cdot t_{12}\cdot t_{14}^{7}\cdot  t_{16}^{4},$$
where $t_{12}\cdot t_{14}^{7}\cdot  t_{16}^{4}$ is a standard $OGS$ elementary element, with \\ $k_1=12, \quad k_2=14, \quad k_3=16$, and $i_{k_1}=1, \quad i_{k_2}=7, \quad i_{k_3}=4$, as it is defined in Definition \ref{elementary}.\\

Since, $14<15<16$ and  $14\leq 14<16$, we consider Proposition \ref{len-pi} on $w_{14}\cdot w_{15}\cdot  (t_{12}\cdot t_{14}^{7}\cdot  t_{16}^{4})$, for the case $k_2\leq L<k_3$.

Then, we get:

$$norm(\pi^{(2)}) =   (s_{12} \cdot s_{11} \cdots s_{2}) \cdot (s_{13} \cdot s_{12}  \cdots s_{3}) $$ $$\cdot (s_{14} \cdot s_{13} \cdots s_{2} \cdot s_{1^{\prime}}\cdot s_1\cdot s_2 \cdots s_{10}) \cdot (s_{15} \cdot s_{14}  \cdots s_{2} \cdot s_{1^{\prime}}\cdot s_1\cdot s_2 \cdots s_{11})  .$$

The length of $\pi^{(2)}$:
$$\ell(\pi^{(2)}) = \ell(t_{12}\cdot t_{14}^{7}\cdot  t_{16}^{4}) + 2\cdot (14-4) + 2\cdot (15-4).$$

Now, consider $\pi^{(3)}$.
$$\pi^{(3)}=w_{20}\cdot t_{19}^{10}\cdot t_{22}^{5}\cdot  t_{23}^{4},$$
where $t_{19}^{10}\cdot t_{22}^{5}\cdot  t_{23}^{4}$ is a standard $OGS$ elementary element, with \\ $k_1=19, \quad k_2=22, \quad k_3=23$, and $i_{k_1}=10, \quad i_{k_2}=5, \quad i_{k_3}=4$, as it is defined in Definition \ref{elementary}.\\

Since, $19<20<22$, we consider Proposition \ref{len-pi} on $w_{20}\cdot  (t_{19}^{10}\cdot t_{22}^{5}\cdot  t_{23}^{4})$, for the case $k_1\leq L<k_2$.

Then, we get:

$$norm(\pi^{(3)})=(s_{19}\cdot s_{18} \cdots s_{11})\cdot (s_{20}\cdot s_{19}\cdots s_2\cdot s_{1^{\prime}}\cdot s_1\cdot s_2 \cdots s_{11})$$ $$\cdot (s_{21}\cdot s_{20}\cdots s_{13})\cdot (s_{22}\cdot s_{21}\cdot s_{20}\cdot s_{19}).$$

The length of $\pi^{(3)}$:
$$\ell(\pi^{(3)}) = \ell(t_{19}^{10}\cdot t_{22}^{5}\cdot  t_{23}^{4}) + 2\cdot (20-9).$$

Therefore,

$$norm(\pi)=norm(\pi^{(1)}\cdot \pi^{(2)}\cdot\pi^{(3)})=norm(\pi^{(1)})\cdot norm(\pi^{(2)})\cdot  norm(\pi^{(3)}).$$

Hence,

$$\ell(\pi)=\ell(\pi^{(1)}\cdot \pi^{(2)}\cdot\pi^{(3)})=\ell(\pi^{(1)})+\ell(\pi^{(2)})+\ell(\pi^{(3)}).$$
\\

Now, we find~ $\pi=\pi^{\bullet}\cdot \pi^{\circ}$ ~decomposition (where, $\pi^{\bullet}\in Id_n^{\bullet}$ and $\pi^{\circ}\in S_n^{\circ}$) of $\pi$, as it is described in Remark \ref{wt-general},  by considering Equation \ref{factorization-dn}.\\

By applying Lemma \ref{tw-standard} on $(t_{7}^{2}\cdot t_{9}^{2}\cdot  t_{12}^{3})\cdot w_{14}\cdot w_{15}$ we get:

$$\pi = w_{11}\cdot w_{14} \cdot w_{15} \cdot t_{7}^{2}\cdot t_{9}^{2}\cdot  t_{12}^{3} \quad |  t_{12} \cdot t_{14}^{7} \cdot  t_{16}^{4} \quad | w_{20} \cdot t_{19}^{10} \cdot  t_{22}^{5} \cdot t_{23}^{4}$$

Then, by applying Lemma \ref{tw-standard} on $(t_{12} \cdot t_{14}^{7} \cdot  t_{16}^{4})\cdot w_{20}$.

$$\pi = w_{11}\cdot w_{14} \cdot w_{15} \cdot t_{7}^{2}\cdot t_{9}^{2}\cdot  t_{12}^{3} \quad |w_{20} \cdot  t_{12} \cdot t_{14}^{7} \cdot  t_{16}^{4} \quad |  t_{19}^{10} \cdot  t_{22}^{5} \cdot t_{23}^{4}$$

Then, by applying Lemma \ref{tw-standard} on $(t_{7}^{2}\cdot t_{9}^{2}\cdot  t_{12}^{3})\cdot w_{20}$.

$$\pi = w_{11}\cdot w_{14} \cdot w_{15} \cdot  w_{20} \cdot t_{7}^{2}\cdot t_{9}^{2}\cdot  t_{12}^{3} \quad | t_{12} \cdot t_{14}^{7} \cdot  t_{16}^{4} \quad |  t_{19}^{10} \cdot  t_{22}^{5} \cdot t_{23}^{4}$$

Hence,
$$\pi=\pi^{\bullet}\cdot \pi^{\circ}$$
such that:
$$\pi^{\bullet}=w_{11}\cdot w_{14} \cdot w_{15} \cdot  w_{20},$$

$$\pi^{\circ}=t_{7}^{2}\cdot t_{9}^{2}\cdot  t_{12}^{3} \quad | t_{12} \cdot t_{14}^{7} \cdot  t_{16}^{4} \quad |  t_{19}^{10} \cdot  t_{22}^{5} \cdot t_{23}^{4}. $$

\end{example}

\section{Conclusions and future plans}
This paper is a natural continuation of the paper \cite{S1}, where  there was introduced a quite interesting generalization of the fundamental theorem for abelian groups to two important and very elementary families of non-abelian Coxeter groups, the $I$-type (dihedral groups), and the $A$-type (symmetric groups).
There were introduced canonical forms, with very interesting exchange laws, and quite interesting properties concerning the Coxeter lengths of the elements. In this paper we generalized the results of \cite{S1} for the $D$-type Coxeter groups, namely $D_n$. The results of the paper  motivate us for further generalizations of the $OGS$ and the arising properties from it for more families of  Coxeter and generalized Coxeter groups, which have an importance in the classification
of Lie algebras and the Lie-type simple groups, and in other fields of mathematics, such as
algebraic geometry for classification of fundamental groups of Galois covers of surfaces \cite{alst}.
Furthermore, it is interesting to find generalization of the $OGS$ to the affine classical families $\tilde{A}_n$, $\tilde{B}_n$, $\tilde{C}_n$, and $\tilde{D}_n$, and also to other generalizations of the mentioned Coxeter groups, as the complex reflection groups $G(r, p, n)$ \cite{ST} or the generalized affine classical groups, the definition of which is described in \cite{rtv}, \cite{ast}.

\end{document}